\documentclass[11pt]{article}
\usepackage{definitions}
    
\begin{document}

\title{On the Infimal Sub-differential Size of Primal-Dual Hybrid Gradient Method \blue{and Beyond}}

\author{Haihao Lu\thanks{Corresponding author, The University of Chicago, Booth School of Business (haihao.lu@chicagobooth.edu).} \and Jinwen Yang\thanks{The University of Chicago, Department of Statistics (jinweny@uchicago.edu).}}

\date{March 2023}

\maketitle

\begin{abstract}
    Primal-dual hybrid gradient method (PDHG, a.k.a. Chambolle and Pock method~\cite{chambolle2011first}) is a well-studied algorithm for minimax optimization problems with a bilinear interaction term. Recently, PDHG is used as the base algorithm for a new LP solver PDLP that aims to solve large LP instances by taking advantage of modern computing resources, such as GPU and distributed system. Most of the previous convergence results of PDHG are either on duality gap or on distance to the optimal solution set, which are usually hard to compute during the solving process. In this paper, we propose a new progress metric for analyzing PDHG, which we dub infimal sub-differential size (IDS),  by utilizing the geometry of PDHG iterates. IDS is a natural extension of the gradient norm of smooth problems to non-smooth problems, and it is tied with KKT error in the case of LP. Compared to traditional progress metrics for PDHG, such as duality gap and distance to the optimal solution set, IDS always has a finite value and can be computed only using information of the current solution. We show that IDS monotonically decays, and it has an $\mathcal O(\frac{1}{k})$ sublinear rate for solving convex-concave primal-dual problems, and it has a linear convergence rate if the problem further satisfies a regularity condition that is satisfied by applications such as linear programming, quadratic programming, TV-denoising model, etc.  Furthermore, we present examples showing that the obtained convergence rates are tight for PDHG. The simplicity of our analysis and the monotonic decay of IDS suggest that IDS is a natural progress metric to analyze PDHG. As a by-product of our analysis, we show that the primal-dual gap has $\mathcal O(\frac{1}{\sqrt{k}})$ convergence rate for the last iteration of PDHG for convex-concave problems. \blue{The analysis and results on PDHG can be directly generalized to other primal-dual algorithms, for example, proximal point method (PPM), alternating direction method of multipliers (ADMM) and linearized alternating direction method of multipliers (l-ADMM)}. 
    
\end{abstract}

\section{Introduction}
Linear programming (LP)~\cite{bertsimas1997introduction,dantzig2002linear} is a fundamental problem in mathematical optimization and operations research. The traditional LP solvers are essentially based on either simplex method or barrier method. However, it is highly challenging to further scale up these two methods. The computational bottleneck of both methods is solving linear equations, which does not scale well on modern computing resources, such as GPUs and distributed systems. Recently, a numerical study~\cite{applegate2021practical} demonstrates that an enhanced version of Primal-Dual Hybrid Gradient (PDHG) method, called PDLP\footnote{The solver is open-sourced at \href{https://developers.google.com/optimization}{Google OR-Tools}. }, can reliably solve LP problems to high precision. As a first-order method, the computational bottleneck of PDHG is matrix-vector multiplication, which can efficiently leverage GPUs and distributed system. It sheds light on solving huge-scale LP where the data need to be stored in a distributed system.

This work is motivated by this recent line of research on using PDHG for solving large-scale LP. The convergence guarantee of PDHG, for generic convex-concave problems or for LP, has been extensively studied in the literature~\cite{zhu2008efficient,chambolle2011first,chambolle2016ergodic,alacaoglu2019convergence,applegate2021faster}. Most of the previous convergence analyses on PDHG are based on either the primal-dual gap or the distance to the optimal solution setting. However, it can be highly non-trivial to compute these metrics while running the algorithms. For primal-dual gap, the iterated solutions when solving LP are often infeasible to the primal and/or the dual problem until the algorithm identifies an optimal solution; thus, the primal-dual gap at any iterates may likely be infinity. For the distance to the optimal solution set, it is not computable unless the optimal solution set is known. Furthermore, even if these values can be computed, they often oscillate dramatically in practice, which makes it hard to evaluate the solving progress. Motivated by these issues, the goal of this paper is to answer the following question:

\vspace{0.3cm}
\centerline{\textit{Is there an evaluatable and monotonically decaying progress metric for analyzing PDHG?}}

We provide an affirmative answer to the above question by proposing a new progress metric for analyzing PDHG, which we dub infimal sub-differential size (IDS). IDS essentially measures the distance between $0$ and the sub-differential set of the objective. It is a natural extension of the gradient norm to potentially non-smooth problems.  Compared with other progress measurements for PDHG, such as primal-dual gap and distance to optimality, IDS always has a finite value and can be computable directly only using the information of the current solution without the need of knowing the optimal solution set. More importantly, IDS decays monotonically during the solving process, making it a natural progress measurement for PDHG. The design of IDS also take into consideration the geometry of PDHG by using a natural norm of PDHG.

Furthermore, most previous PDHG analyzes~\cite{chambolle2011first,chambolle2016ergodic} focus on the ergodic rate, that is, the performance of the average PDHG iterates. On the other hand, it is frequently observed that the last iteration of PDHG has comparable and sometimes superior performance compared to the average iteration in practice~\cite{applegate2021practical}. Our analysis on IDS provides an explanation of such behavior. We show that IDS of the last iterate of PDHG has $\mathcal O(1/k)$ convergence rate for convex-concave problems. As a direct byproduct, the primal dual gap at the last iterate has $\mathcal O(1/\sqrt{k})$ convergence rate, which is inferior to the $\mathcal O(1/k)$ convergence of the average iterate~\cite{chambolle2016ergodic}. This explains why the average iterate can be a good option. Furthermore, we show that if the problem further has metric sub-regularity, a condition that is satisfied by many optimization problems, IDS and primal dual gap both enjoy linear convergence, whereas average iterates can often have sublinear convergence. This, \textcolor{black}{accompanied with other results under metric sub-regularity in literature,} explains the numerical observation that the last iteration may have superior behaviors.



More formally, we consider a generic minimax problem with a bilinear interaction term: 
\begin{equation}\label{eq:minmax}
    \min_{x\in \mathbb R^n}\max_{y\in \mathbb R^m} \mathcal L(x,y)=f(x)+\left \langle Ax,y \right \rangle-g(y)
\end{equation}
where $f:\mathbb R^n\rightarrow (-\infty,+\infty]$, $g:\mathbb R^m\rightarrow (-\infty,+\infty]$ are simple proper lower semi-continuous convex functions and $A\in \mathbb R^{m\times n}$ is a matrix. Note that functions $f(x)$ and $g(y)$ can encode the constraints $x\in\mathcal X$ and $y\in \mathcal Y$ by using an indicator function. Throughout the paper, we assume that there is a finite optimal solution to the minimax problem \eqref{eq:minmax}.The primal-dual form of linear programming can be formulated as an instance of \eqref{eq:minmax}:
$$\min_{x\ge 0}\max_y \mathcal L(x,y)=c^Tx+y^TAx-b^Ty\ .$$

Beyond LP, problems of form \eqref{eq:minmax} are ubiquitous in statistics and machine learning. For example,
\begin{itemize}
    \item In image processing, total variation (TV) denoising model~\cite{bredies2015tgv} plays a central role in variational methods for imaging and it can be reformulated as $\min_{x\ge 0}\max_{y:\Vert y \Vert_{\infty}\leq 1} \mathcal L(x,y)=\frac{\lambda}{2}\Vert x-f\Vert_2^2+y^T\nabla x$, where $\nabla$ is the differential matrix.
    \item In statistics, LASSO~\cite{tibshirani1996regression, chen1994basis} is a popular regression model, whose primal-dual formulation is $\min_{x}\max_y\mathcal L(x,y)=\lambda \Vert x \Vert_1+y^TAx-(\frac 12 \Vert y \Vert_2^2 +b^Ty)$.
    \item In machine learning, support vector machine (SVM)~\cite{cortes1995support,boser1992training} is a classic supervised learning method for classification and can be formulated as $\min_{x}\max_{ -1/n\le y_i/b_i\le 0}\mathcal L(x,y)=\frac{\lambda}{2}\Vert x\Vert_2^2+y^TAx-\frac 1n\sum_{i=1}^n\frac{ny_i}{b_i}$.
\end{itemize}


Primal-dual hybrid gradient method (PDHG, a.k.a. Chambolle and Pock algorithm \cite{chambolle2011first}) described in Algorithm \ref{alg:pdhg} is one of the most popular algorithms for solving the structured minimax problem \eqref{eq:minmax}. It has been extensively used in the field of image processing: total-variation image denoising \cite{bredies2015tgv}, multimodal medical imaging \cite{knoll2016joint}, computation
of nonlinear eigenfunctions \cite{gilboa2016nonlinear}, and many others. More recently, PDHG also serves as the base algorithm for a new large-scale LP solver PDLP \cite{applegate2021practical, applegate2021infeasibility,applegate2021faster}. 


\begin{algorithm}
    \renewcommand{\algorithmicrequire}{\textbf{Input:}}
    \caption{Primal Dual Hybrid Gradient (PDHG)\protect\footnotemark}
    \label{alg:pdhg}
    \begin{algorithmic}[1]
        \REQUIRE Initial point $(x_0,y_0)$, step-size $s >0$.
        \FOR{$k=0,1,...$}
        \STATE $x_{k+1}=\text{prox}_{f}^{s}(x_k-s A^Ty_k)$
        \STATE $y_{k+1}=\text{prox}_{g}^{s}(y_k+s A(2x_{k+1}-x_k))$
        \ENDFOR
    \end{algorithmic}
\end{algorithm}
\footnotetext{PDHG is often presented in a form with different primal and dual step sizes~\cite{chambolle2011first,chambolle2016ergodic}. Here, we choose to use the same primal and dual step size for notational convenience. Our results can easily extend to the case of different step sizes by rescaling the variable.}

The contribution of the paper can be summarized as follows:
\begin{itemize}
    \item \blue{We propose to identify the ``right'' progress metric when studying an optimization algorithm. We propose a new progress metric IDS for analyzing PDHG, and show that it monotonically decays (see Section \ref{sec:sec-decay-sublinear}).} 
    \item We show that IDS converges to $0$ with $\mathcal O(1/k)$ convergence rate when solving a convex-concave minimax problem \eqref{eq:minmax} (Theorem \ref{thm:thm-decay}). As a by-product, this shows $\mathcal O(1/\sqrt{k})$ non-ergodic convergent rate of the duality gap for PDHG when solving convex-concave minimax problems.
    \item We show that IDS converges to $0$ with a linear convergence rate if the problem satisfies metric sub-regularity, which is satisfied by many applications.
    \blue{\item We extend the above results of PDHG to other classic algorithms for minimax problems, including proximal point method, alternating direction method of multipliers and linearized alternating direction method of multipliers. 
    \item The proofs of the above results are surprisingly simple, which is in contrast to the recent literature on the last iterate convergence for cousin algorithms.}
\end{itemize}

\subsection{Related literature}
\textbf{Convex-concave minimax problems.} There have been several lines of research that develop algorithms for convex-concave minimax problems and general monotone variational inequalities. The two most classic algorithms may be proximal point method (PPM) and extragradient method (EGM). Rockafellar introduces PPM for monotone variational inequalities \cite{rockafellar1976monotone} and, around the same time, Korpelevich proposes EGM for solving the convex-concave minimax problem \cite{korpelevich1976extragradient}. Convergence analysis on PPM and EGM has been flourishing since then. Tseng proves the linear convergence for PPM and EGM on strongly-convex-strongly-concave minimax problems and on unconstrained bilinear problems \cite{tseng1995linear}. Nemirovski shows that EGM, as a special case of mirror-prox algorithm, exhibits $\mathcal O(\frac 1\epsilon)$ sublinear rate for solving general convex-concave minimax problems over a compact and bounded set \cite{nemirovski2004prox}. Moreover, the connection between PPM and EGM is strengthened in \cite{nemirovski2004prox}, i.e., EGM approximates PPM. 
Another line of research investigates minimax problems with a bilinear interaction term~\eqref{eq:minmax}. For such problems, the two most popular algorithms are perhaps PDHG \cite{chambolle2011first} and Douglas-Rachford splitting (DRS) \cite{douglas1956numerical, eckstein1992douglas}.
Recently, motivated by application in statistical and machine learning, there has been renewed interest in minimax problems. In particular, for bilinear problem ($f(x)=g(y)=0$ in \eqref{eq:minmax}), \cite{daskalakis2018training,mokhtari2020unified} show that the Optimistic Gradient Descent Ascent (OGDA) enjoys a linear convergence rate with full rank $A$ and \cite{mokhtari2020unified} proves that EGM and PPM also converge linearly under the same condition. Besides, \cite{mokhtari2020unified} builds up the connection between OGDA and PPM: OGDA is an approximation to PPM for bilinear problems. Under the framework of ODE, \cite{lu2020s} investigates the dynamics of unconstrained primal-dual methods and yields tight condition under which different algorithms exhibit linear convergence.

\textbf{Primal-dual hybrid gradient method (PDHG).} The early works on PDHG were motivated by applications in image processing and computer vision~\cite{zhu2008efficient,condat2013primal,esser2010general,he2012convergence,chambolle2011first}. Recently, PDHG is the base algorithm used in a large-scale LP solver~ \cite{applegate2021practical,applegate2021infeasibility,applegate2021faster}. The first convergence guarantee of PDHG was in the average iteration and was proposed in~\cite{chambolle2011first}. Later, \cite{chambolle2016ergodic} presented a simplified and unified analysis of the ergodic rate of PDHG. More recently, many variants of PDHG have been proposed, including adaptive version~\cite{vladarean2021first, malitsky2018first,pock2011diagonal,goldstein2015adaptive} and stochastic version~\cite{chambolle2018stochastic,alacaoglu2019convergence,lu2021linear}. PDHG was also shown to be equivalent to DRS \cite{o2020equivalence,liu2021acceleration}.

\textbf{Ergodic rate vs non-ergodic rate.} In the literature on convex-concave minimax problems, the convergence is often described in the ergodic sense, i.e., on the average of iterates. For example, Nemirovski \cite{nemirovski2004prox} shows that EGM (or more generally the mirror prox algorithm) has $\mathcal O(1/k)$ primal-dual gap for the average iterate; Chambolle and Pock \cite{chambolle2011first, chambolle2016ergodic} show that PDHG has $\mathcal O(1/k)$ primal-dual gap on the average iteration. More recently, the convergence rate of last iterates (i.e., non-ergodic rate) attracts much attention in the machine learning and optimization community due to its practical usage. Most of the works on last iterates study the gradient norm when the problem is differentiable. For example, \cite{golowich2020last,gorbunov2021extragradient} show that the squared norm of gradient of iterates from EGM converges at $\mathcal O(1/k)$ sublinear rate. \textcolor{black}{ \cite{yoon2021accelerated} proposes an extra anchored gradient descent algorithm converges with faster rate $\mathcal O(1/k^2)$ when the objective function is smooth. The recent works~\cite{cai2022tight,gorbunov2022last} study the last iteration convergence of EGM and OGDA in the constrained setting for EGM and OGDA. More specifically,  \cite{cai2022tight} utilizes the sum-of-squares (SOS) programming to analyze the tangent residual and \cite{gorbunov2022last} uses semi-definite programming (SDP) and Performance Estimation Problems (PEPs) to derive the last iteration bounds of a complicated potential function for the two algorithms. The tangent residual can be viewed as a special case of IDS when $f$ and $g$ are the sum of a smooth function and an indicator function, and IDS can handle general non-smooth terms.} 

\textcolor{black}{Another related line of works is the study of squared distance between successive iterates for splitting schemes~\cite{davis2016convergence,davis2017faster}. Although there are connections between IDS and distance between successive iterates, the geometric interpretations are quite different: IDS is a direct quality metric to characterize the size of the sub-differential at current iterate while the distance between successive iterates measures fixed-point residual, which is an indirect characterization of the quality of the solution. In terms of the primal-dual gap, \cite{golowich2020last} shows that the last iterate of EGM has $\mathcal O(1/\sqrt{k})$ rate, which is slower than the average iterate of EGM, and \cite{davis2016convergence,fercoq2019coordinate} shows $\mathcal O(1/\sqrt{k})$ rate of objective errors for some primal-dual methods.}  As a by-product of our analysis, we show that the last iterate of PDHG also has $\mathcal O(1/\sqrt{k})$ rate in the primal-dual gap. 

{On the other hand, linear convergence of primal-dual algorithms are often observed in the last iterate.} To obtain the linear convergence rate, additional conditions are required. There are three different types of conditions studied in the literature: strong-convexity-concavity~\cite{tseng1995linear} (i.e. strong monotoncity), interaction dominance/negative comonotonicity~\cite{grimmer2020landscape,lee2021fast}, and metric sub-regularity~\cite{latafat2019new,liang2016convergence,alacaoglu2019convergence,fercoq2021quadratic}. The metric used in the above analysis is usually the norm of gradient for differentiable problems or the distance to optimality. The linear convergence under additional assumptions partially explains the behavior that practitioners may choose to favor the last iterate over the average iterate.

\textbf{Metric sub-regularity.} The notion of metric sub-regularity is extensively studied in the community of variational analysis~\cite{ioffe1979necessary,ioffe2008metric,dontchev2004regularity,kruger2015error,zheng2007metric,zheng2010metric,zheng2014metric,dontchev2009implicit}. It was first introduced in~\cite{ioffe1979necessary} and later the terminology ``metric sub-regularity'' was suggested in~\cite{dontchev2004regularity}, which is also closely related to calmness~\cite{henrion2002calmness}. For more detail in this direction, see~\cite{dontchev2009implicit}. 

The regularity condition holds for many important applications. For example, \cite{zheng2014metric,latafat2019new} show that the problem with piecewise linear quadratic functions on a compact set satisfies the regularity condition, including Lasso and support vector machines, etc. Partially due to this reason, recently there arises extensive interest in analyzing first-order methods under the assumption of metric sub-regularity. For convex minimization, \cite{drusvyatskiy2018error} shows that the metric sub-regularity of sub-gradient is equivalent to the quadratic error bound. The results for primal-dual algorithms are also fruitful ~\cite{latafat2019new,liang2016convergence,alacaoglu2019convergence,fercoq2021quadratic,lu2020adaptive}. In the context of PDHG, under metric sub-regularity, \cite{latafat2019new,fercoq2021quadratic} proves the linear rate of deterministic PDHG and \cite{alacaoglu2019convergence} shows the linear convergence of stochastic PDHG in terms of distance to the optimal set.

\subsection{Notations} 
We use $\Vert\cdot\Vert_2$ to denote Euclidean norm and $\left \langle \cdot, \cdot \right \rangle$ for its associated inner product. For a positive definite matrix $M\succ 0$, we denote $\left \langle \cdot, \cdot \right \rangle_M=\left \langle \cdot, M\cdot \right \rangle$. Let $\Vert\cdot\Vert_M$ be the norm induced by the inner product $\left \langle \cdot, \cdot \right \rangle_M$. Let $dist_M(z,\mathcal Z)$ be the distance between point $z$ and set $\mathcal Z$ under the norm $\Vert \cdot \Vert_M$, that is, $dist_M(z,\mathcal Z)=\min_{u\in \mathcal Z}\Vert u \Vert_M$. Denote $\mathcal{Z}^*$ as the optimal solution set to \eqref{eq:minmax}. Let $\sigma_{min}^+(A)$ be the minimum nonzero singular value of a matrix $A$. Denote $\Vert A \Vert$ the operator norm of a matrix $A$. Let $\iota_{\mathcal C}(\cdot)$ be the indicator function of the set $\mathcal C$. $f(x)=\mathcal O(g(x))$ denote that for sufficiently large $x$, there exists constant $C$ such that $f(x)\leq Cg(x)$ and $f(x)=\Omega(g(x))$ denote that for sufficiently large $x$, there exists constant $c$ such that $f(x)\geq cg(x)$. The proximal operator is defined as $\text{prox}_h^{\tau}(x):=\text{argmin}_{y\in \mathbb R^n} \left \{h(y)+\frac{1}{2\tau} \Vert y-x \Vert^2 \right \}$. Denote $z=(x,y)$, and $\mathcal F(z)=\mathcal F(x,y)=\begin{pmatrix} \partial f(x)+A^Ty \\ -Ax+\partial g(y) \end{pmatrix}$ as the sub-differential of the objective. \blue{ For a matrix $A$, denote $A^-$ a pseudo-inverse of $A$ and $\mathrm{range}(A)$ the range of $A$. }

\section{Infimal sub-differential size (IDS)}\label{sec:sec-decay-sublinear}

In this section, we introduce a new progress measurement, which we dub infimal sub-differential size (IDS), for PDHG. IDS is a natural extension of the squared norm of gradient for smooth optimization problems to non-smooth optimization problems. Compared to other progress measurements for PDHG, such as primal-dual gap and distance to optimality, IDS always has a finite value and is computable directly without the need to know the optimal solution set. More importantly, unlike the other metric that may oscillate over time, the IDS monotonically decays along the iteration of the PDHG, further suggesting that the IDS is a natural progress measurement for PDHG. Here is the formal definition of IDS:
\begin{mydef}\label{def:def-distance}
The infimal sub-differential size (IDS) at solution $z$ for PDHG with step-size $s$ is defined as:
    \begin{equation}\label{eq:IDS}
        dist_{P_s^{-1}}^2(0,\mathcal F(z)) \ ,
    \end{equation}
where $P_s=\begin{pmatrix}
\frac 1s I_n & -A^T \\ -A & \frac 1s I_m
\end{pmatrix}$ is the PDHG norm and $\mathcal F(z)=\begin{pmatrix} \partial f(x)+A^Ty \\ -Ax+\partial g(y) \end{pmatrix}$ is the sub-differential of the objective $\mathcal L(x,y)$.
\end{mydef}

Note that the first-order optimality condition of \eqref{eq:minmax} is $0\in \mF(z)$. In other words, if the original minimax problem \eqref{eq:minmax} is feasible and bounded, the minimization problem
\begin{equation}\label{eq:iss-minimization}
    \min_{z\in \RR^{m+n}} dist_{P_s^{-1}}^2(0,\mathcal F(z))
\end{equation}
has optimal value $0$, and furthermore \eqref{eq:iss-minimization} and \eqref{eq:minmax} share the same optimal solution set, thus IDS is a valid progress measurement.


\begin{figure}
\centering
\begin{subfigure}{0.4\textwidth}
  \centering
  \includegraphics[scale=0.35]{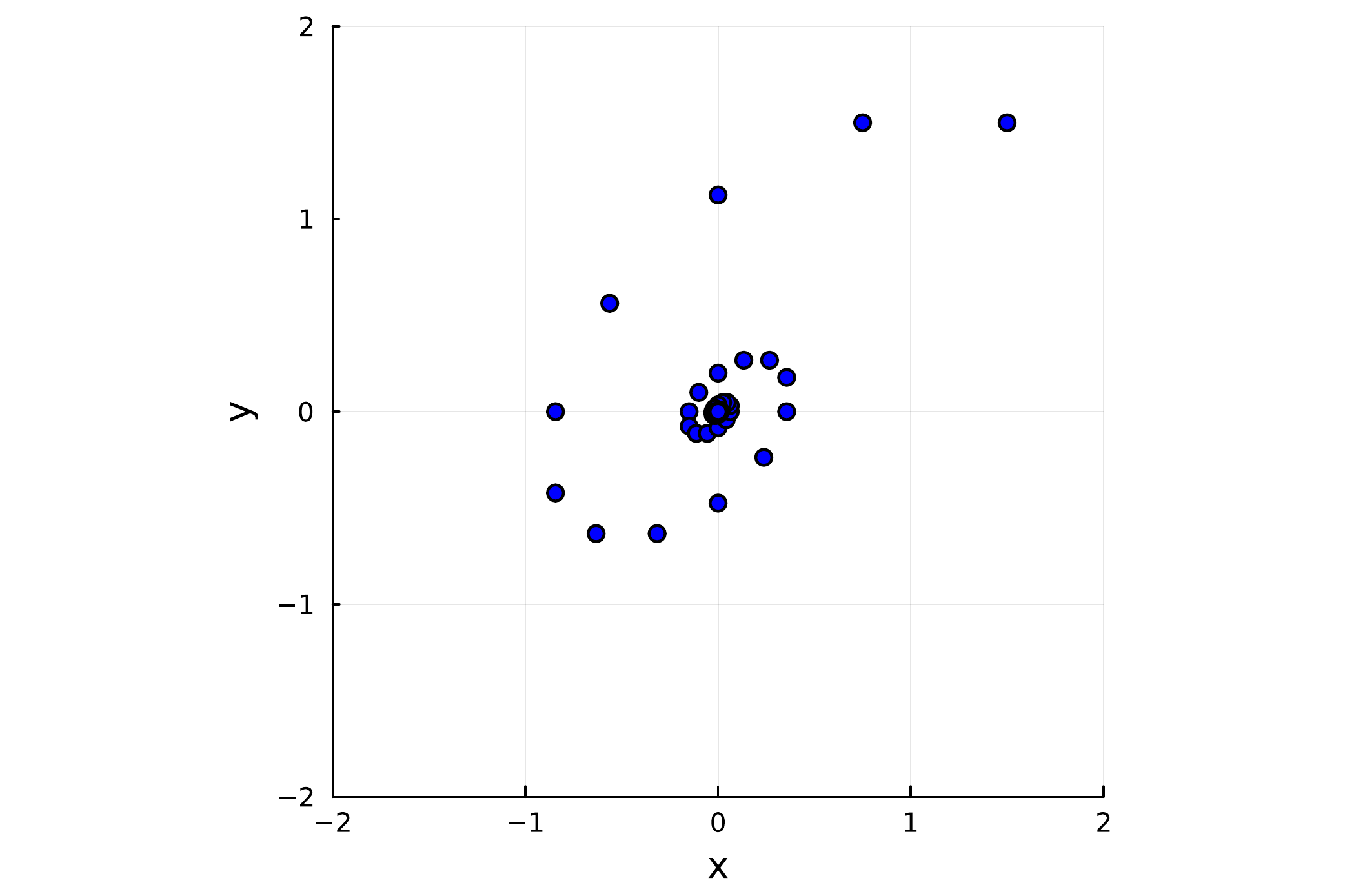}
  \caption{PDHG}
  \label{fig:PDHG}
\end{subfigure}%
\begin{subfigure}{0.4\textwidth}
  \centering
  \includegraphics[scale=0.35]{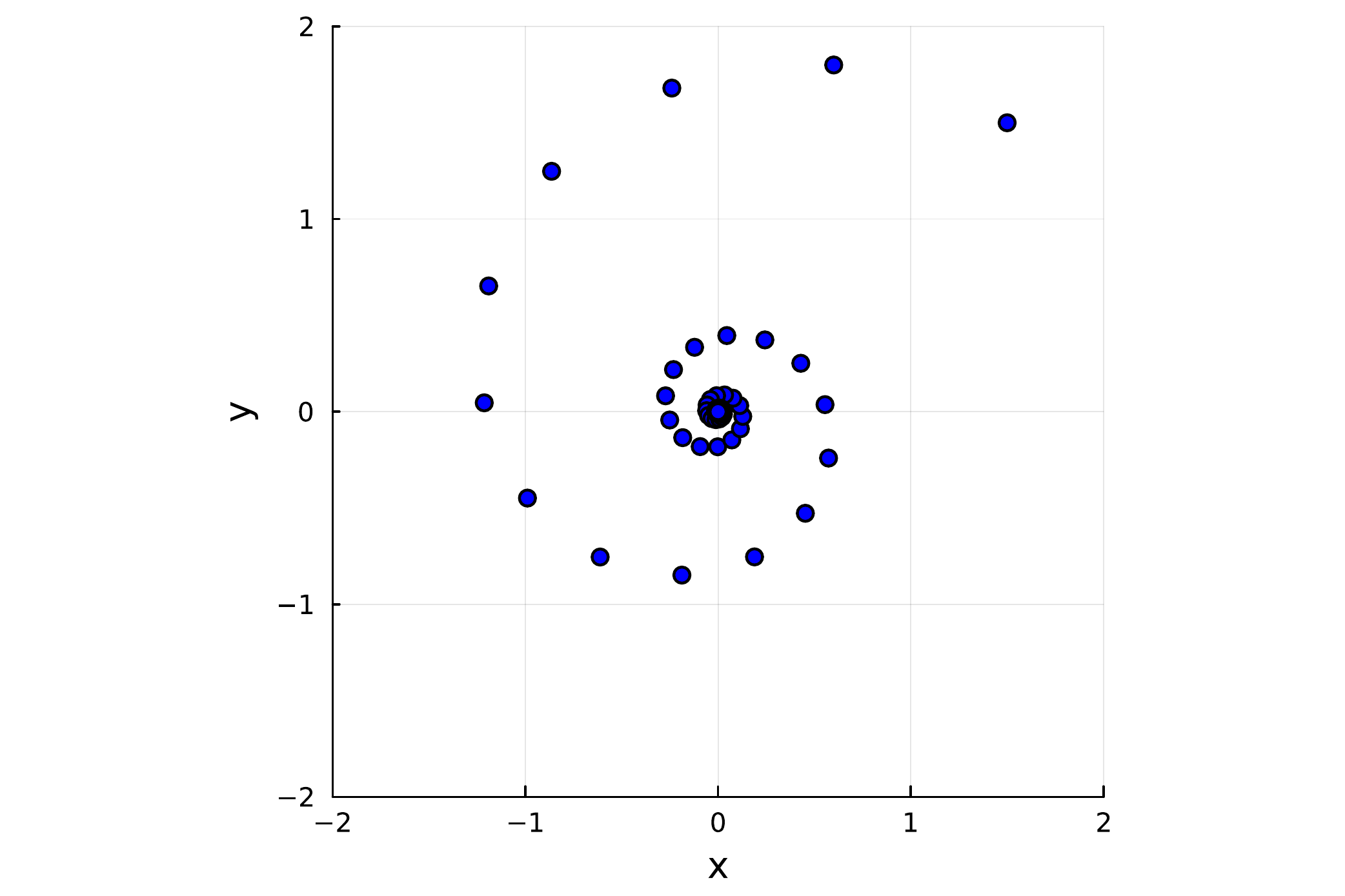}
  \caption{PPM}
  \label{fig:PPM}
\end{subfigure}
\caption{The trajectory of PDHG and PPM to solve a simple unconstrained bilinear problem $\min_{x}\max_{y} xy$.}
\label{fig:bilinear}
\end{figure}

When measuring the distance between $0$ and the set $\mF(z)$, a crucial part is the choice of the norm. The simplest norm which is also used in practice is the $\ell_2$ norm. However, this may not be the natural norm for all algorithms. To understand the intuition, Figure \ref{fig:bilinear} plots the trajectory of PDHG and proximal point method (PPM) to solve a simple unconstrained bilinear problem 
\begin{equation*}
    \min_{x\in \mathbb R}\max_{y\in \mathbb R} \mathcal L(x,y):=xy \ .
\end{equation*}
We can observe that the iterations of two algorithms exhibit spiral patterns toward the unique saddle point $(0,0)$. Nevertheless, in contrast to the circular pattern of PPM, the spiral of PDHG spins in a tilting elliptical manner. The fundamental reason for such an elliptical structure is the asymmetricity of the primal and dual steps in PDHG.  The elliptical structures of the iterates can be captured by the choice of norms in the analysis. This is consistent with the fact that while the analysis of PPM~\cite{rockafellar1976monotone} utilizes $\ell_2$ norm, the analysis of PDHG~\cite{chambolle2016ergodic,applegate2021faster,he2012convergence} often utilizes a different norm defined by $P_s=\begin{pmatrix}
\frac 1s I_n & -A^T \\ -A & \frac 1s I_m
\end{pmatrix}$. \textcolor{black}{The $P_s$ has already been used in analyzing PDHG, for example, \cite{he2012convergence} studies the contraction of PDHG under $P_s$ norm and motivate variants of PDHG from this perspective. We here propose to use $P_s^{-1}$ in the sub-differentiable (dual) space. } 

In previous works, the convergence guarantee of PDHG is often stated with respect to the duality gap~\cite{chambolle2011first,chambolle2016ergodic}, $\max_{\hat x\in \mathcal X, \hat y\in \mathcal Y}\mathcal L(x,\hat y)-\mathcal L(\hat x,y)$, or the distance to the optimal solution set~\cite{applegate2021faster,fercoq2021quadratic}, $dist(z,\mathcal Z^*)$. Unfortunately, these values are barely computable or informative because the iterates are usually infeasible until the end of the algorithms (thus, the duality gap at the iterates is often infinity) and the optimal solution set is unknown to the users beforehand. In practice, people instead use residual as a performance metric because of its computational efficiency.  For example, the KKT residual, i.e., the violation of primal feasibility, dual feasibility, and complementary slackness, is widely used in the LP solver as the performance metric. \textcolor{black}{
Indeed, the KKT residual of LP is upper bounded by IDS upto a constant. More formally,
consider the primal-dual form of standard LP
\begin{equation*}
    \min_{x\geq 0}\max_y\; c^Tx-y^TAx+b^Ty\ ,
\end{equation*}
and the corresponding KKT residual for $z=(x,y)$, $x\geq 0$ is given by $\left\Vert \begin{pmatrix}
        Ax-b \\ [A^Ty-c]^+ \\  [c^Tx-b^Ty]^+
    \end{pmatrix}  \right\Vert_2$. \\
    Then it holds for any $z=(x,y)$, $x\geq 0$ that
    \begin{equation*}
        dist^2(0,\mathcal F(z)) \geq \rho_r^2(z) \geq \frac{1}{1+4R^2}\left\Vert \begin{pmatrix}
        Ax-b \\ [A^Ty-c]^+ \\  [c^Tx-b^Ty]^+
    \end{pmatrix}  \right\Vert_2^2 \ ,
    \end{equation*}
    where $R$ is an upper bound for the norm of iterates $z_k$ and $\rho_r(z)$ is the normalized duality gap defined in \cite[Equation (4a)]{applegate2021faster}. The first inequality follows from \cite[Proposition 5]{applegate2021faster} and the second inequality uses \cite[Lemma 4]{applegate2021faster}. Therefore, IDS under $l_2$ norm is an upper bound of KKT residual for LP. Since all norms in finite-dimensional Hilbert space are equivalent (upto a constant), this showcases that the IDS under $P_s$ norm also provides an upper bound on KKT residual.
    }

Since $\mF(z)$ is in the dual space, we utilized the dual norm of $P_s$ in the definition of IDS. In fact, for the simple bilinear problem in Figure \ref{fig:bilinear}, the use of $P_s^{-1}$ norm is the key to guarantee the monotonic decay of IDS for PDHG iterates. Indeed, this monotonicity result holds for any convex-concave problems and can be formalized in the below proposition: 
\begin{prop}\label{thm:thm-decay}
Consider PDHG (Algorithm \ref{alg:pdhg}) with step-size $s < \frac{1}{\Vert A \Vert}$ for solving a convex-concave minimax problem \eqref{eq:minmax}. It holds for $k\geq 0$ that
\begin{equation*}
    dist_{P_s^{-1}}^2(0,\mathcal F(z_{k+1}))\leq dist_{P_s^{-1}}^2(0,\mathcal F(z_k)) \ .
\end{equation*}
\end{prop}

\begin{proof}
First, notice that the iterate update of Algorithm \ref{alg:pdhg} can be written as
\begin{equation}\label{eq:update}
    P_s(z_k-z_{k+1})\in \mathcal F(z_{k+1}) \ .
\end{equation}
\blue{Note that PDHG is a forward-backward algorithm \cite{combettes2014forward}, we can rewrite the proximal operators in Algorithm \ref{alg:pdhg} as the following} 
\begin{align}\label{eq:iterates}
    \begin{split}
        & \ 0\in \partial f(x_{k+1})+\frac 1s (x_{k+1}-x_k+sA^Ty_k) \\ 
        & \ 0 \in\partial g(y_{k+1})+\frac 1s (y_{k+1}-y_k-sA(2x_{k+1}-x_k)) \ .
    \end{split}
\end{align}
We arrive at \eqref{eq:update} by rearranging \eqref{eq:iterates}:
\begin{align*}
    \begin{split}
    P_s(z_k-z_{k+1}) & \ =\begin{pmatrix}
        \frac 1s I_n & -A^T \\ -A & \frac 1s I_m
    \end{pmatrix}\begin{pmatrix}
        x_k-x_{k+1} \\ y_k-y_{k+1}
    \end{pmatrix}=\begin{pmatrix}
        \frac 1s(x_k-x_{k+1})-A^T(y_k-y_{k+1}) \\  -A(x_k-x_{k+1})+\frac 1s(y_k-y_{k+1})
    \end{pmatrix}\\
    & \ \in \begin{pmatrix} \partial f(x_{k+1})+A^Ty_{k+1} \\ -Ax_{k+1}+\partial g(y_{k+1}) \end{pmatrix}=\mathcal F(z_{k+1}) \ .
    \end{split}
\end{align*}

Now denote $\widetilde\omega_{k+1}=P_s(z_k-z_{k+1})\in \mathcal F(z_{k+1})$, then it holds for any $\omega_k\in\mathcal F(z_k)$ that
\begin{align*}
    \begin{split}
        \Vert \widetilde\omega_{k+1} \Vert_{P_s^{-1}}^2-\Vert \omega_k \Vert_{P_s^{-1}}^2 & = \left \langle \widetilde\omega_{k+1}-\omega_k, P_s^{-1}(\widetilde\omega_{k+1}+\omega_k)\right \rangle \\
        &\leq \left \langle \widetilde\omega_{k+1}-\omega_k, 2(z_{k+1}-z_k)+P_s^{-1}(\widetilde\omega_{k+1}+\omega_k)\right \rangle\\
        & = \left \langle P_s(z_k-z_{k+1})-\omega_k, 2(z_{k+1}-z_k)+P_s^{-1}(P_s(z_k-z_{k+1})+\omega_k)\right \rangle\\
        & = \left \langle P_s(z_k-z_{k+1})-\omega_k, -(z_k-z_{k+1})+P_s^{-1}\omega_k\right \rangle\\
        & = -\Vert z_k-z_{k+1} \Vert_{P_s}^2+2\omega_k^T(z_k-z_{k+1})-\Vert \omega_k \Vert_{P_s^{-1}}^2\\
        & = -\Vert P_s(z_k-z_{k+1})-\omega_k \Vert_{P_s^{-1}}^2\\
        & = -\Vert \widetilde\omega_{k+1}-\omega_k \Vert_{P_s^{-1}}^2 \\
        & \leq 0 \ ,
    \end{split}
\end{align*}
where the inequality utilizes $\mF$ is a monotone operator due to the convexity-concavity of the objective $\mathcal{L}(x,y)$, thus $\left \langle \widetilde\omega_{k+1}-\omega_k, z_{k+1}-z_k\right\rangle\ge 0$, and the second equality uses $\widetilde\omega_{k+1}=P_s(z_k-z_{k+1})$.
Now choose $\omega_k=\arg\min_{w\in \mathcal F(z_k)} \{ \|w\|_{P_s^{-1}}^2\}$, and we obtain
\begin{equation*}
    dist_{P_s^{-1}}^2(0,\mathcal F(z_{k+1}))\leq \Vert \widetilde\omega_{k+1} \Vert_{P_s^{-1}}^2 \leq \Vert \omega_k \Vert_{P_s^{-1}}^2 = dist_{P_s^{-1}}^2(0,\mathcal F(z_k)) \ ,
\end{equation*}
which finishes the proof.
\end{proof}


We end this section by making two remarks on the proof of Proposition \ref{thm:thm-decay}:
\begin{rem}
The proof of Proposition \ref{thm:thm-decay} not only shows the monotonicity of IDS, but also provides a bound on its decay: 
\begin{equation}\label{eq:eq-strict}
    dist_{P_s^{-1}}^2(0,\mathcal F(z_{k+1})) \leq dist_{P_s^{-1}}^2(0,\mathcal F(z_{k}))  -\Vert \widetilde\omega_{k+1}-\omega_k \Vert_{P_s^{-1}}^2 \ , 
\end{equation}
where $\omega_k=\arg\min_{w\in \mathcal F(z_k)} \{ \|w\|_{P_s^{-1}}^2\}\in \mathcal F(z_k)$ and $\widetilde\omega_{k+1}=P_s(z_k-z_{k+1})\in \mathcal F(z_{k+1})$.
\end{rem}

\begin{rem}
    In the proof of Proposition \ref{thm:thm-decay} and the convergence proof in Section \ref{sec:sublinear} and Section \ref{sec:linear}, the only information we use from PDHG is the update rule \eqref{eq:update} and the fact that $P_s$ is a positive definite matrix. This makes it easy to generalize the results of PDHG to other algorithms. For example, PPM to solve the convex-concave minimax problem $\min_{x}\max_{y} \mathcal L(x,y)$ has the following update rule:
    \begin{equation*}
    \frac 1s(z_k-z_{k+1})\in \mathcal F(z_{k+1})\ ,
    \end{equation*}
    and the iterate update is also a special case of \eqref{eq:update} where $P_s=\frac 1s I$ and $\mathcal F(z)=\begin{pmatrix} \partial_x \mathcal{L}(x,y) \\ -\partial_y \mathcal L(x,y) \end{pmatrix}$. Thus, all of the analysis and results we develop herein for PDHG can be directly extended to PPM in parallel.
\end{rem}





\section{Sublinear convergence of IDS}\label{sec:sublinear}
In this section, we present the sublinear convergence rate of IDS. Our major result is stated in Theorem \ref{thm:thm-sublinear}. As a direct consequence, the result implies sublinear rate of non-ergodic iterates on duality gap.


\begin{thm}\label{thm:thm-sublinear}
Consider the iterates $\{z_k\}_{k=0}^{\infty}$ of PDHG (Algorithm \ref{alg:pdhg}) to solve a convex-concave minimax problem \eqref{eq:minmax}. Let $z_*\in \mathcal Z^*$ be an optimal point to \eqref{eq:minmax}. Suppose the step-size satisfies $s<\frac{1}{\Vert A \Vert}$. Then, it holds for any iteration {$k\geq 1$} that
\begin{equation*}
        dist_{P_s^{-1}}^2(0,\mathcal F(z_{k}))\leq  \frac{1}{k} \Vert z_{0}-z_* \Vert_{P_s}^2 \ .
\end{equation*}
\end{thm}

\begin{proof}
Recall the update of PDHG \eqref{eq:update}. Then we have for any $i$ that:
\begin{equation*}
    \widetilde\omega_{i+1}:= P_s(z_i-z_{i+1})\in \mathcal F(z_{i+1}) \ .
\end{equation*}
It holds for any iteration $i\geq 0$ that
\begin{align*}
    \begin{split}
        \Vert z_i-z_* \Vert_{P_s}^2 & \ =\Vert z_i-z_{i+1}+z_{i+1}-z_* \Vert_{P_s}^2\\
        & \ =\Vert z_{i}-z_{i+1} \Vert_{P_s}^2+\Vert z_{i+1}-z_* \Vert_{P_s}^2+2\left\langle z_{i+1}-z_*,P_s(z_i-z_{i+1})\right\rangle\\
        & \ = \Vert \omega_{i+1} \Vert_{P_s^{-1}}^2+\Vert z_{i+1}-z_* \Vert_{P_s}^2+2\left\langle z_{i+1}-z_*,\widetilde\omega_{i+1}-0\right\rangle \\
        & \ \geq \Vert z_{i+1}-z_* \Vert_{P_s}^2+ \Vert \widetilde\omega_{i+1} \Vert_{P_s^{-1}}^2 \ ,
    \end{split}
\end{align*}
where the inequality comes from the monotonicity of operator $\mathcal F$ and notices $0\in \mathcal{F} (z_*)$. Thus, we have 
\begin{equation}\label{eq:eq-sublinear}
    dist_{P_s^{-1}}^2(0,\mathcal F(z_{i+1})) \leq \Vert \widetilde\omega_{i+1} \Vert_{P_s^{-1}}^2 \leq \Vert z_i-z_* \Vert_{P_s}^2-\Vert z_{i+1}-z_* \Vert_{P_s}^2 \ .
\end{equation}

Therefore, it follows that
\begin{align*}
    \begin{split}
        dist_{P_s^{-1}}^2(0,\mathcal F(z_{k})) & \  \leq \frac{1}{k}\sum_{i=1}^{k} dist_{P_s^{-1}}^2(0,\mathcal F(z_i)) \\
        & \ \leq \frac{1}{k}\pran{ \sum_{i=1}^{k} \pran{\Vert z_{i-1}-z_* \Vert_{P_s}^2-\Vert z_{i}-z_* \Vert_{P_s}^2}}\\
        & \ \leq  \frac{1}{k} \Vert z_{0}-z_* \Vert_{P_s}^2 \ ,
    \end{split}
\end{align*}
where the first inequality follows from Proposition \ref{thm:thm-decay} and the second inequality is due to inequality \eqref{eq:eq-sublinear}. 
\end{proof}





As a direct consequence of Theorem \ref{thm:thm-sublinear}, we can obtain $\mathcal O(\frac{1}{\sqrt k})$ convergence rate on the primal-dual gap for the last iteration of PDHG. The last iteration of PDHG has a slower convergence rate compared to the average iteration, which has $\mathcal O(\frac{1}{k})$ rate on the primal-dual gap~\cite{chambolle2016ergodic}. \blue{This is consistent with the recent discovery that the last iteration has a slower convergence than the average iteration in a smooth convex-concave saddle point for the extragradient method~\cite{golowich2020last}, Douglas-Rachford splitting~\cite{davis2017faster,davis2016convergence}, and primal-dual coordinate descent method~\cite{fercoq2019coordinate}. }


\begin{cor}\label{cor:cor-sublinear}
Under the same assumption and notation as Theorem \ref{thm:thm-sublinear}, it holds for any iteration $k\ge 1$, $z=(x,y)\in \mathcal Z$ and any optimal solution $z_*$ that
\begin{equation*}
    \mathcal L(x_k,y)-\mathcal L(x,y_k) \leq \frac{1}{\sqrt k}\pran{\Vert z_0-z_* \Vert_{P_s}^2+\Vert z_{0}-z_* \Vert_{P_s}\Vert z_*-z \Vert_{P_s}} \ .
\end{equation*}
\end{cor}
\begin{proof}
Denote $u_k\in \partial f(x_k)$, $v_k\in\partial g(y_k)$ and $w_k=\begin{pmatrix} u_k+A^Ty_k \\ v_k-Ax_k \end{pmatrix}\in \mathcal F(z_k)$. Then
\begin{align}\label{eq:gap-IDS}
    \begin{split}
        \mathcal L(x_k,y)-\mathcal L(x,y_k) & \ = \mathcal L(x_k,y)-\mathcal L(x_k,y_k)+\mathcal L(x_k,y_k)-\mathcal L(x,y_k) \\
        & \ \leq  (-v_k+Ax_k)^T(y-y_k)+(u_k+A^Ty_k)^T(x_k-x) \\ 
        & \ = \left\langle w_k, z_k-z \right\rangle \leq \Vert w_k \Vert_{P_s^{-1}} \Vert z_k-z \Vert_{P_s} \ ,
    \end{split}
\end{align}
where the first inequality uses the convexity of $f$ and $g$ and the last one follows from Cauchy-Schwarz inequality.
\textcolor{black}{Notice  \eqref{eq:gap-IDS} holds for any $w_k\in \mathcal F(z_k)$, we can choose $\omega_k=\arg\min_{\omega\in \mathcal F(z_k)}\{\|\omega\|_{P_s^{-1}}^2  \}$ and thus $\|\omega_k\|_{P_s^{-1}}^2=dist_{P_s^{-1}}^2(0,\mathcal F(z_k))$. Then it holds that}
\begin{align*}
    \begin{split}
        \mathcal L(x_k,y)-\mathcal L(x,y_k) & \ \leq dist_{P_s^{-1}}(0,\mathcal F(z_k)) \Vert z_k-z \Vert_{P_s} \\
        & \ \leq \sqrt{\frac{1}{k}}\Vert z_{0}-z_* \Vert_{P_s}\Vert z_k-z \Vert_{P_s}\\
        & \ \leq \frac{1}{\sqrt k}\Vert z_{0}-z_* \Vert_{P_s}\pran{\Vert z_k-z_* \Vert_{P_s}+\Vert z_*-z \Vert_{P_s}}\\
        & \ \leq \frac{1}{\sqrt k}\Vert z_{0}-z_* \Vert_{P_s}\pran{\Vert z_0-z_* \Vert_{P_s}+\Vert z_*-z \Vert_{P_s}} \ .
    \end{split}
\end{align*}
where the second inequality is due to Theorem \ref{thm:thm-sublinear}; the third inequality follows from the triangle inequality, and
the last inequality uses $\Vert z_k-z_* \Vert_{P_s}\leq \Vert z_0-z_* \Vert_{P_s}$.
\end{proof} 


\section{Linear convergence of IDS}\label{sec:linear}
In previous sections, we show that the IDS monotonically decays along iterates of PDHG and has a sublinear convergence rate for a convex-concave minimax problem. In this section, we further show that the IDS has linear convergence if the minimax problem further satisfies a regularity condition. This regularity condition is satisfied by many applications, such as unconstrained bilinear problem, linear programming, strongly-convex-strongly-convex problems, etc.

We begin by introducing the following regularity condition.
\begin{mydef}\label{def:def-linear-condition}
We say that the minimax problem \eqref{eq:minmax}, or equivalently the operator $\mathcal F$, satisfies metric sub-regularity with respect to matrix $P_s$ on set $\mathbb S$ if it holds for any $z\in \mathbb S$ that
\begin{equation}\label{eq:eq-linear-condition}
    \alpha_s dist_{P_s}(z,\mathcal Z^*) \leq dist_{P_s^{-1}}(0,\mathcal F(z)) \ .
\end{equation}
\end{mydef}

\begin{rem}\label{rem:rem-subreg}
We comment that the metric sub-regularity condition we defined above is a special case of the traditional definition of metric sub-regularity~\cite{dontchev2009implicit} with a set-valued function $F(z)$ and a vector $0$. 

\end{rem}

Recently, the metric sub-regularity and related sharpness conditions have been used to analyze primal-dual algorithms~\cite{fercoq2021quadratic,lu2020adaptive,alacaoglu2019convergence, applegate2021faster,yang2016linear,yuan2020discerning}. In particular, it has been shown that strongly-convex-strongly-concave problems satisfy this condition globally, and piecewise linear quadratic functions satisfy this condition on a compact set~\cite{zheng2014metric,latafat2019new}. This includes many examples, such as LASSO, support vector machine, linear programming and quadratic programming on a compact region in the primal-dual formulation. Suppose the step-size satisfies $s\le \frac{1}{2\|A\|}$, one can show that unconstrained bilinear problem and generic linear program satisfy the regularity condition \eqref{eq:eq-linear-condition}. Below we list a few concrete examples that satisfy metric subregularity:

\begin{exam}[Unconstrained bilinear problem]\label{exam:bilinear}
The unconstrained bilinear problem 
\begin{equation}\label{eq:bilinear}
    \min_{x\in \mathbb R^n}\max_{y\in \mathbb R^m} c^Tx+y^TAx-b^Ty 
\end{equation}
satisfies global metric sub-regularity with $\alpha_s=\frac{s}{2}\sigma_{min}^+(A)$ and $\mathbb S=\mathbb R^{m+n}$.
\end{exam}

\begin{exam}[Linear programming]\label{exam:lp}
For any optimal solution $z_*$, the primal-dual form of linear programming
\begin{equation}\label{eq:lp}
    \min_{x\in \mathbb R^n}\max_{y\in \mathbb R^m} c^Tx+\iota_{\mathbb R_+^n}(x)+y^TAx-b^Ty
\end{equation}
satisfies the metric sub-regularity condition on a bounded region $\mathbb S=\{z:\|z-z_*\|_{P_s}\le \|z_0-z_*\|_{P_s}\}$ with $\alpha_s=\frac{s}{2H(K)\sqrt{1+4R^2}}$, where $H(K)$ is the Hoffman constant of the KKT system of the LP and $R=2\Vert z_0-z_* \Vert+\Vert z_* \Vert$. Furthermore, the iterates of PDHG from initial solution $z_0$ stay in the bounded region $\mathbb S$.
\end{exam}

\begin{exam}[Strongly-convex-strongly-concave problem]\label{exam:scsc}
Consider minimax problem: 
\begin{equation}\label{eq:strongly}
    \min_{x\in \mathbb R^n}\max_{y\in \mathbb R^m}  f(x)+\left \langle Ax,y \right \rangle-g(y)
\end{equation}
where $f:\mathbb R^n\rightarrow \mathbb R$, $g:\mathbb R^m\rightarrow \mathbb R$ are $\mu$-strongly convex functions. Then, the problem satisfies the metric sub-regularity condition with $\alpha_s=\frac{s\mu}{4}$.
\end{exam}

The proofs of these examples follow directly from \cite{applegate2021faster,fercoq2021quadratic}. 

The next theorem presents the linear convergence of IDS for PDHG under metric sub-regularity:
\begin{thm}\label{thm:thm-linear}
    Consider the iterations $\{z_k\}_{k=0}^{\infty}$ of PDHG (Algorithm \ref{alg:pdhg}) to solve a convex-concave minimax problem \eqref{eq:minmax}. Suppose the step-size $s<  \frac{1}{\Vert A \Vert}$, and the minimax problem \eqref{eq:minmax} satisfies metric sub-regularity condition \eqref{eq:eq-linear-condition} on a set $\mathbb S$ that contains $\{z_k\}_{k=0}^{\infty}$. Then, it holds for any iteration $k\ge \left\lceil e/\alpha_s^2\right\rceil$ that
    \begin{equation*}
        dist_{P_s^{-1}}^2(0,\mathcal F(z_k))\leq \exp\pran{1-\frac{k}{\left\lceil e/\alpha_s^2\right\rceil}}dist_{P_s^{-1}}^2(0,\mathcal F(z_0)) \ .
    \end{equation*}
\end{thm}

\begin{proof}
For any iteration $k\geq 1$, suppose $c\left\lceil \frac{e}{\alpha_s^2}\right\rceil \leq k < (c+1)\left\lceil \frac{e}{\alpha_s^2}\right\rceil$ for a non-negative integer $c$.
It follows from Theorem \ref{thm:thm-sublinear} that for any $z_*\in \mathcal{Z}^*$, we have
\begin{align}\label{eq:eq-linear-key-1}
\begin{split}
    dist_{P_s^{-1}}^2(0,\mathcal F(z_k))& \ \leq \frac{1}{{\left\lceil e/ \alpha_s^2 \right \rceil}}\Vert z_{{k-\left\lceil e/ \alpha_s^2 \right \rceil}}-z_* \Vert_{P_s}^2 .
    \end{split}
\end{align}
Taking the minimum over $z_*\in \mathcal{Z}^*$ in the RHS of \eqref{eq:eq-linear-key-1}, we obtain
\begin{align}\label{eq:eq-linear-key}
\begin{split}
    dist_{P_s^{-1}}^2(0,\mathcal F(z_k))& \ \leq \frac{1}{{\left\lceil e/ \alpha_s^2 \right \rceil}}dist_{P_s}^2(z_{{k-\left\lceil e/ \alpha_s^2 \right \rceil}},\mathcal Z^* ) \ \leq \frac{1}{\alpha_s^2 {\left\lceil e/ \alpha_s^2 \right \rceil}}dist_{P_s^{-1}}^2(0,\mathcal F(z_{{k-\left\lceil e/ \alpha_s^2 \right \rceil}})) \\ & \ \leq \frac{1}{ { e}}dist_{P_s^{-1}}^2(0,\mathcal F(z_{{k-\left\lceil e/ \alpha_s^2 \right \rceil}})) \ ,
    \end{split}
\end{align}
where the second inequality utilizes condition \eqref{eq:eq-linear-condition}. Recursively using \eqref{eq:eq-linear-key}, we arrive at
\begin{align*}
    \begin{split}
        dist_{P_s^{-1}}^2(0,\mathcal F(z_k))& \ \leq \frac 1e dist_{P_s^{-1}}^2(0,\mathcal F(z_{k-\left\lceil e/ \alpha_s^2 \right \rceil})) \leq\cdot\cdot\cdot \leq  \pran{\frac 1 e}^{c}dist_{P_s^{-1}}^2(0,\mathcal F(z_{k-c\left\lceil e/ \alpha_s^2 \right \rceil}))\\
        & \ \leq \exp(-c)dist_{P_s^{-1}}^2(0,\mathcal F(z_0))\\
        & \ \leq \exp\pran{1-\frac{k}{\left\lceil e/\alpha_s^2\right\rceil}}dist_{P_s^{-1}}^2(0,\mathcal F(z_{0})) \ .
    \end{split}
\end{align*}
where the last two inequalities are due to the monotonicity of $dist_{P_s^{-1}}^2(0,\mathcal F(z_k))$ (see Proposition \ref{thm:thm-decay}) and $c\left\lceil \frac{e}{\alpha_s^2}\right\rceil \leq k < (c+1)\left\lceil \frac{e}{\alpha_s^2}\right\rceil$, respectively.
\end{proof}

\begin{rem}
Theorem \ref{thm:thm-linear} implies that we need in total $\mathcal O\pran{\pran{\frac{1}{\alpha_s}}^2\log\pran{\frac 1 \epsilon}}$ iterations to find an $\epsilon$-close solution $z$ such that $dist_{P_s^{-1}}^2(0,\mathcal F(z))\leq\epsilon$.
\end{rem}

\begin{rem}
Similar to Corollary \ref{cor:cor-sublinear}, it is straight-forward to derive the linear convergence rate on duality gap for the last iterate of PDHG under metric sub-regularity condition from Theorem \ref{thm:thm-linear}. In particular, it holds that
\begin{align*}
    \begin{split}
        \mathcal L(x_k,y)-\mathcal L(x,y_k) \leq \exp\pran{-\frac{k-\left\lceil e/\alpha_s^2\right\rceil}{2\left\lceil e/\alpha_s^2\right\rceil}}dist_{P_s^{-1}}(0,\mathcal F(z^{0})) \pran{dist_{P_s}(z_0, \mathcal Z^*)+\Vert z_*-z \Vert_{P_s}} \ .
    \end{split}
\end{align*}
\end{rem}

\section{Tightness}
In this section, we demonstrate that the derived sublinear and linear rates of the last iteration of PDHG in previous sections are tight, i.e., there exists an instance on which the obtained convergence rate cannot be improved. \blue{The tightness herein does not refer to a lower bound result, i.e., this argument does not rule out primal-dual first-order algorithms that can achieve even better performance in the decay of IDS.}

\subsection{Linear convergence}

To begin with, Theorem \ref{thm:thm-lb-linear} presents a convex-concave instance that satisfies the metric sub-regularity condition \eqref{eq:eq-linear-condition}, on which the last iteration of PDHG requires at least $\Omega\pran{\pran{\frac{1}{\alpha_s}}^2\log\frac 1\epsilon}$ iterations to identify a solution such that $dist_{P_s^{-1}}^2(0,\mathcal F(z))\leq \epsilon$. Combined with Theorem \ref{thm:thm-linear}, we conclude that for problems satisfying condition \eqref{eq:eq-linear-condition} the linear rate derived in Theorem \ref{thm:thm-linear} is tight up to a constant.
\begin{thm}\label{thm:thm-lb-linear}
For any iteration $k\geq 1$ and step-size $s\leq \frac{1}{2\Vert A \Vert}$, there exist an initial point $z_0$ and a convex-concave minimax problem \eqref{eq:minmax} satisfying the metric sub-regularity condition \eqref{eq:eq-linear-condition} with $\alpha_s<\frac{1}{2}$, such that the PDHG iterates $\{z_k\}$ satisfy
 \begin{equation}\label{eq:eq-thm-tight-sublinear}
     dist_{P_s^{-1}}^2(0,\mathcal F(z_k)) \geq \frac{1}{12}\pran{1-4\alpha_s^2}^{k} dist_{P_s^{-1}}^2(0,\mathcal F(z_0)) \ .
\end{equation}
\end{thm}

    

To construct tight instances, we utilize the following intermediate result from  \cite[Section 5.2]{applegate2021faster}: 
\begin{prop}\label{prop:prop-tight}
Consider a 2-dimensional vector sequence $\{a_k\}$ such that
\begin{align*}
    \begin{split}
        a_{k+1}=\begin{pmatrix}
            1-2s^2\sigma^2 & -s\sigma\\
            s\sigma & 1
        \end{pmatrix}a_k \ ,
    \end{split}
\end{align*}
with $\sigma>0$ and $s \sigma < 1/2$. Then, it holds for any $k\geq 0$ that
\begin{equation*}
    \left\Vert a_k \right\Vert_2^2 \geq \frac 13 (1-s^2\sigma^2)^k \left\Vert a_0 \right\Vert_2^2 \ .
\end{equation*}
\end{prop}

\begin{proof}[Proof of Theorem \ref{thm:thm-linear}]
Let $f(x)=g(y)=0$, $x\in \mathbb R^m,y\in \mathbb R^m$ and $A=\text{diag}(\sigma_1, ..., \sigma_m)\in \mathbb R^{m\times m}$ with $0<\sigma_1\leq ...\leq \sigma_m$. Set the initial point $z_0=e_1=(1,0,\ldots,0)\in \mathbb R^{m+n}$ as the first standard basis vector. Then, we have  $$dist_{P_s^{-1}}^2(0,\mathcal F(z_0))=\Vert F(z_0) \Vert_{P_s^{-1}}^2\geq \frac s2 \Vert F(z_0) \Vert_2^2 > 0\ ,$$
where the first inequality uses $P_s^{-1}\succeq \frac{s}{2}I$. Furthermore, the PDHG update can be rewritten following from \eqref{eq:update} as
\begin{equation*}
    \begin{pmatrix}
        I & -sA^T \\ -sA & I
    \end{pmatrix}
    z_{k+1}=\begin{pmatrix}
        I & -sA^T \\ -sA & I
    \end{pmatrix}z_k-s
    \begin{pmatrix}
        0 & A^T \\ -A & 0
    \end{pmatrix}z_{k+1} \ ,
\end{equation*}

\textcolor{black}{
    which is equivalent to the update rule
    \begin{equation}\label{eq:update-lb}
    z_{k+1}=\begin{pmatrix}
        I & 0 \\ -2sA & I
    \end{pmatrix}^{-1}\begin{pmatrix}
        I & -sA^T \\ -sA & I
    \end{pmatrix}z_k = \begin{pmatrix}
        I & -sA^T \\ sA & I-2s^2AA^T
    \end{pmatrix}z_k \ ,
    \end{equation}
}
\textcolor{black}{
thus it holds that
\begin{align*}
\begin{split}
    F(z_{k+1})=\begin{pmatrix}
        0 & A^T \\ -A & 0
    \end{pmatrix}
    z_{k+1}& \ = \begin{pmatrix}
        0 & A^T \\ -A & 0
    \end{pmatrix}\begin{pmatrix}
        I & -sA^T \\ sA & I-2s^2AA^T
    \end{pmatrix}z_k \\    
    & \ =\begin{pmatrix}
        0 & A^T \\ -A & 0
    \end{pmatrix}\begin{pmatrix}
        I & -sA^T \\ sA & I-2s^2AA^T
    \end{pmatrix}\begin{pmatrix}
        0 & -A^{-1} \\ A^{-T} & 0
    \end{pmatrix}\begin{pmatrix}
        0 & A^T \\ -A & 0
    \end{pmatrix}z_k\\
    & \ =\begin{pmatrix}
        I-2s^2A^TA & -sA^T \\ sA & I
    \end{pmatrix}\begin{pmatrix}
        0 & A^T \\ -A & 0
    \end{pmatrix}z_k=\begin{pmatrix}
        I-2s^2A^TA & -sA^T \\ sA & I
    \end{pmatrix}F(z_k) \ .
\end{split}
\end{align*}
where in the second equality we plug in the update rule \eqref{eq:update-lb} and the third equality uses the non-singularity of matrix $A$. 
}

Notice that $A$ is a diagonal matrix, thus we have
\begin{align}\label{eq:eq-tight}
    \begin{split}
        \Vert F(z_{k})\Vert_{P_s^{-1}}^2 & \ \geq \frac s2 \Vert F(z_{k}) \Vert_2^2 \\
        & \ \geq \frac s2 \pran{(F(z_k)_1)^2+(F(z_k)_{m+1})^2} \\
        & \ \geq \frac s6 (1-s^2\sigma_1^2)^{k}\pran{(F(z_0)_1)^2+(F(z_0)_{m+1})^2}\\
        & \ =\frac s6 (1-s^2\sigma_1^2)^{k}\Vert F(z_{0}) \Vert_2^2 \\
        & \ \geq \frac{1}{12} (1-s^2\sigma_1^2)^{k} \Vert F(z_{0}) \Vert_{P_s^{-1}}^2 \ ,
    \end{split}
\end{align}
where \textcolor{black}{$F(z)_i$ denotes the $i$-th coordinate of gradient vector $F(z)$.} The first and last inequalities are due to the choice of step-size $s$ such that $\frac s2 I\preceq P_s^{-1} \preceq 2sI$, and the third inequality follows from Proposition \ref{prop:prop-tight} and the fact that $A$ is diagonal.

We finish the proof by noticing for this instance that $\alpha_s=\frac{s}{2}\sigma_{min}^+(A)=\frac{s}{2}\sigma_1$ (see Example 1).
\end{proof}

\subsection{Sublinear convergence}
Next, Theorem \ref{thm:thm-lb-sublinear} presents a convex-concave instance on which PDHG requires at least $\Omega\pran{\frac{1}{\epsilon}}$ to achieve a solution such that $dist_{P_s^{-1}}^2(0,\mathcal F(z))\leq \epsilon$. Combined with Theorem \ref{thm:thm-sublinear}, we conclude that the sublinear rate derived in Theorem \ref{thm:thm-sublinear} is tight upto a constant.

\begin{thm}\label{thm:thm-lb-sublinear}
For any iteration $k \geq 1$ and step-size $s= \frac{1}{c\Vert A \Vert}$ with $c\geq 2$, there exist a convex-concave minimax problem \eqref{eq:minmax} and an initial point $z_0$ such that the PDHG iterates $\{z_k\}$ satisfy
 \begin{equation}\label{eq:eq-thm-tight-linear}
     dist_{P_s^{-1}}^2(0,\mathcal F(z_k)) \geq \frac{1}{48c^2e^{4/c^2}}\frac{\Vert z_0-z_* \Vert_{P_s}^2}{k}\ .
\end{equation}
\end{thm}



\begin{proof}
To prove Theorem \ref{thm:thm-lb-sublinear}, we utilize the same instance constructed in Theorem \ref{thm:thm-lb-linear} with carefully chosen singular values of the matrix $A$. In particular, we set $\sigma_1=\frac{L_A}{\sqrt k}$ and $\sigma_2=...=\sigma_m=L_A$. Then we have $\|A\|=L_A$ and $\alpha_s=\frac{\sigma_1}{c\sigma_m}=\frac {1}{c\sqrt k}$.

It follows from Theorem \ref{thm:thm-lb-linear} that
\begin{equation*}
    dist_{P_s^{-1}}^2(0,\mathcal F(z_k))=\Vert F(z_k) \Vert_{P_s^{-1}}^2 \geq \frac{1}{12}\pran{1-4\alpha_s^2}^{k} \Vert F(z_0) \Vert_{P_s^{-1}}^2 \geq \frac{s}{24}\pran{1-\frac{4}{c^2k}}^{k} \Vert F(z_0) \Vert_2^2 \ ,
\end{equation*}
where the first inequality is from Theorem \ref{thm:thm-lb-linear} and the last inequality uses $P_s^{-1}\succeq \frac s2 I$.
Furthermore, notice that
\begin{equation*}
    \Vert F(z_0) \Vert^2=\left\Vert \begin{pmatrix}
        0 & A^T \\ A & 0
    \end{pmatrix}z_0 \right\Vert_2^2 = z_0^T \begin{pmatrix}
        A^TA & 0 \\ 0 & AA^T
    \end{pmatrix}z_0 \geq \sigma_1^2\Vert z_0 \Vert_2^2 \ ,
\end{equation*}
it thus holds that
\begin{align}\label{eq:lower1}
    \begin{split}
        dist_{P_s^{-1}}^2(0,\mathcal F(z_k)) 
          \geq \frac{s}{24}\pran{1-\frac{4}{c^2k}}^{k} \left(\frac{L_A}{\sqrt k}\right)^2\Vert z_0 \Vert_2^2
         = \frac{s}{24}\pran{1-\frac{4}{c^2k}}^{k} \frac{L_A^2}{k}\Vert z_0-z_* \Vert_2^2 \ .
    \end{split}
\end{align}
Meanwhile, it holds that
{\small
\begin{align}\label{eq:lower2}
    \begin{split}
          \frac{s}{24}\pran{1-\frac{4}{c^2k}}^{k} \frac{L_A^2}{k}\Vert z_0-z_* \Vert_2^2
          \ \geq \frac{s}{24e^{4/c^2}}\frac{L_A^2\Vert z_0-z_* \Vert_2^2}{k} 
          \ \geq \frac{s^2}{48e^{4/c^2}}\frac{L_A^2\Vert z_0-z_* \Vert_{P_s}^2}{k}
          \ =\frac{1}{48c^2e^{4/c^2}}\frac{\Vert z_0-z_* \Vert_{P_s}^2}{k} \ ,
    \end{split}
\end{align}}

where the first inequality uses $\pran{1-\frac{4}{c^2k}}^k\geq e^{-4/c^2}$, the second inequality is from $P_s\preceq \frac 2s I$, and the last equality follows from the choice of step-size $s=\frac{1}{c\Vert A \Vert}\leq \frac{1}{2\Vert A \Vert}$.
We finish the proof by combining \eqref{eq:lower1} and \eqref{eq:lower2}.
\end{proof}

\blue{
\section{Efficient Computation and Numerical Results}
In this section, we present a linear-time algorithm to compute IDS for a general problem \eqref{eq:minmax}, and present preliminary numerical experiments on linear programming that validate our theoretical results.

\subsection{Linear-time computation of IDS}
Computing IDS \eqref{eq:IDS} involves computing the projection of $0$ onto the sub-differential set $\mathcal F(z)$ with respect to the norm $\Vert \cdot \Vert_{P_s^{-1}}$, which may not always be trivial. Here, we present a principle approach to efficiently compute IDS in linear time by using Nesterov's accelerated gradient descent (AGD), under the assumption that the non-smooth part of $\mathcal{F}$ is simple. In our numerical experiments, usually 15 iterations of AGD can lead to a high-resolution solution to IDS. 


Consider the convex-concave minimax problem \eqref{eq:minmax}. The sub-differential at a solution $z$ is given by
\begin{equation*}
    \mathcal F(z)=\mathcal F(x,y)=\begin{pmatrix}
        A^Ty \\ -Ax 
    \end{pmatrix}+\begin{pmatrix}
        \partial f(x) \\ \partial g(y)
    \end{pmatrix} \ .
\end{equation*}
We denote $d(z):=\begin{pmatrix}
        A^Ty \\ -Ax 
    \end{pmatrix}$ and $\mathcal G(z):=\begin{pmatrix}
        \partial f(x) \\ \partial g(y)
    \end{pmatrix}$, then the projection problem can be formulated as solving the following quadratic programming:
\begin{align}\label{eq:effiecient}
    \begin{split}
        \min_{\omega\in \mathcal G} \left\Vert \omega+ d \right\Vert_{P_s^{-1}}^2= \omega^T P_s^{-1}\omega+2d^TP_s^{-1}\omega+d^TP_s^{-1}d  \ .
    \end{split}
\end{align}

For most of the applications of PDHG, the sub-gradient set $\mathcal G(z)$ as well as the projection onto $\mathcal G(z)$ are straight-forward to compute. For example, when $f(x)=\iota_{\mathcal C}(x)$ is an indicator function of closed convex set $\mathcal C$, the  subdifferential is its normal cone, that is, $\partial f(x)=N_{\mathcal C}(x)$. If $f(x)=\|x\|_1$ is $l_1$ norm, then $\partial f(x)=\{ g:\|g\|_{\infty}\leq 1, g^Tx=\|x\|_1 \}$.

Notice that \eqref{eq:effiecient} is a convex quadratic programming. A natural algorithm for solving the QP \eqref{eq:effiecient} is Nesterov's accelerated gradient descent, which we formally present in Algorithm \ref{alg:agd}. Recall that throughout the section, we assume the step size $s$ satisfies $s\leq \frac{1}{2\Vert A \Vert}$. Thus, we have $\frac s2 I\preceq P_s^{-1} \preceq 2sI$ and the condition number of $P_s^{-1}$ is $\kappa(P_s^{-1})\leq 4$. As the result, this QP is easy to solve and the complexity of AGD for solving the problem is  $\mathcal O\pran{\log\frac 1 \epsilon}$ \cite{nesterov2003introductory}. Since the complexity is independent of the dimension of the problem, we have a linear time computation of IDS. In our numerical experiments, we found that it usually takes less than 15 iterations of Nesterov's accelerated methods to solve the subproblem \eqref{eq:effiecient} almost exactly, which further verifies the effectiveness of AGD.



\begin{algorithm}
    \renewcommand{\algorithmicrequire}{\textbf{Input:}}
    \caption{Accelerated Gradient Descent for solving \eqref{eq:effiecient}}
    \label{alg:agd}
    \begin{algorithmic}[1]
        \REQUIRE Initial point $w_0$, step-size $\beta$, condition number $\kappa$.
        \FOR{$k=0,1,...$}
        \STATE $w_{k+1}=\text{Proj}_{\mathcal G}\pran{u_k-\frac 2 \beta P_s^{-1}(u_k+d)}$
        \STATE $u_{k+1}=w_{k+1}+\frac{\sqrt\kappa-1}{\sqrt\kappa+1}(w_{k+1}-w_k)$
        \ENDFOR
    \end{algorithmic}
\end{algorithm}

\subsection{Numerical experiments on PDHG for LP}

Here, we present preliminary numerical results on PDHG for LP that verify our theoretical results in the previous sections. 

\textbf{Datasets.} We utilize the root-node LP relaxation of three datasets from \texttt{MIPLIB}, \texttt{enlight\_hard}, \texttt{supportcase14} and \texttt{nexp-50-20-4-2}. The dimension (i.e. $n$ and $m$) of the three datasets are summarized in Table~\ref{table:agd}. These instances are chosen such that vanilla PDHG converges to an optimal solution within a reasonable number of iterations, and similar results hold for other instances.

\textbf{Implementation details.} We run PDHG (Algorithm \ref{alg:pdhg}) on the three instances from MIPLIB with step-size $s=\frac{1}{2\Vert A\Vert}$. We calculate IDS using AGD (Algorithm \ref{alg:agd}). We terminate Algorithm \ref{alg:agd} when the first order optimality condition of \eqref{eq:effiecient} holds with $10^{-10}$ tolerance (this effectively means we solve the subproblem \eqref{eq:effiecient} up to the machine precision). We summarize the average iterations of AGD in Table \ref{table:agd}.

\textbf{Results.} Figure \ref{fig:size} presents IDS and KKT residual versus number of iterations of PDHG for the three LP instances. To make them at the same scale, we square KKT residual (recall that the definition of IDS also has a square). As we can see, IDS monotonically decays, which verifies Proposition \ref{thm:thm-decay}. In contrast, there can be a huge fluctuation in KKT residual, which showcases that IDS is a more natural progress metric for PDHG. Furthermore, as we can see, the decay of IDS follows a linear rate, which verifies our theoretical results in Theorem \ref{thm:thm-linear}. 


\begin{table}
\centering
\begin{tabular}{ |c|c|c| } 
 \hline
 Instance & Size & Average iteration of AGD \\  \hline
 \texttt{enlight\_hard} & $100\times 200$ & 12.6  \\ \hline
 \texttt{supportcase14} & $234\times 304$ & 15.0  \\
 \hline
 \texttt{nexp-50-20-4-2} & $540\times 1225$ & 13.04  \\
 \hline
\end{tabular}
\caption{Test instance sizes and average iteration numbers of Algorithm \ref{alg:agd}}
\label{table:agd}
\end{table}

\begin{figure}
\centering
\begin{subfigure}{0.32\textwidth}
  \centering
  \includegraphics[width=\linewidth]{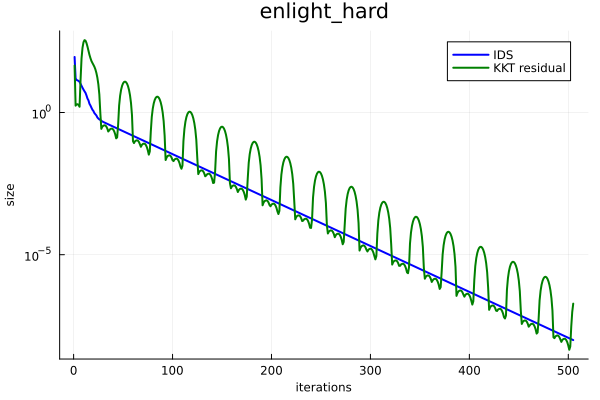}
  \caption{enlight\_hard}
  \label{fig:enlighthard}
\end{subfigure}
\begin{subfigure}{0.32\textwidth}
  \centering
  \includegraphics[width=\linewidth]{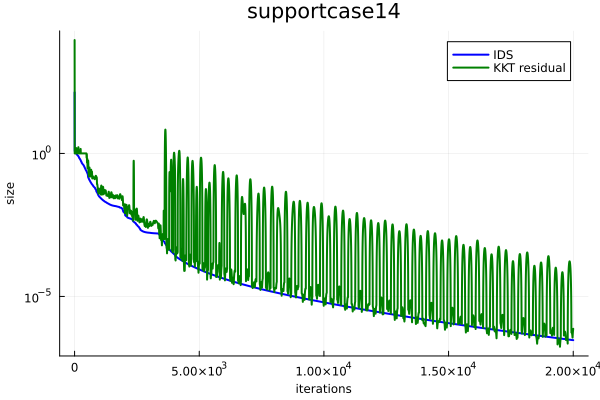}
  \caption{supportcase14}
  \label{fig:supportcase14}
\end{subfigure}
\begin{subfigure}{0.32\textwidth}
  \centering
  \includegraphics[width=\linewidth]{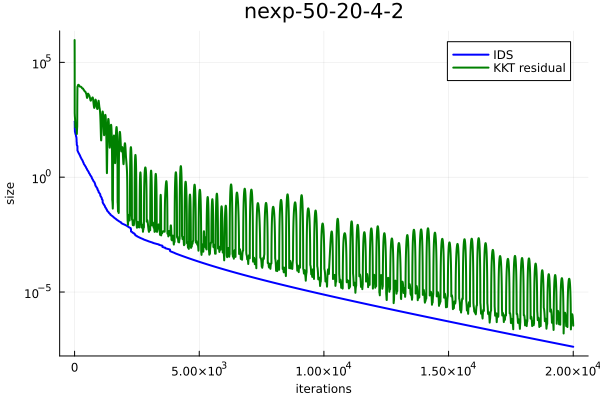}
  \caption{nexp-50-20-4-2}
  \label{fig:nexp}
\end{subfigure}
\caption{Infimal sub-differential size v.s. iterations}
\label{fig:size}
\end{figure}


\section{Extensions to Other Algorithms}
The results on IDS in Section  \ref{sec:sec-decay-sublinear}-\ref{sec:linear} are not limited to PDHG. In this section, we showcase how to extend these results to three classic algorithms, proximal point methods (PPM), alternating direction method of multipliers (ADMM) and linearized ADMM (l-ADMM).

The cornerstone of the analysis in Section \ref{sec:sec-decay-sublinear}-\ref{sec:linear} is to identify a positive definite matrix $\mathcal P$ such that the update rule for iterate $z_{k+1}$ has the form
\begin{equation}\label{eq:cornerstone}
    \mathcal P(z_k-z_{k+1}) \in \mathcal F(z_{k+1}) \ .
\end{equation}
This is the only fact of PDHG that we used in our analysis. Indeed,
if an algorithm can be formulated as \eqref{eq:cornerstone}, we can define IDS as
\begin{equation}\label{eq:general-IDS}
        dist_{\mathcal P^{-1}}^2(0,\mathcal F(z)) \ ,
\end{equation}
and the corresponding monotonicity (Proposition \ref{thm:thm-decay}) and complexity analysis of IDS (Theorem \ref{thm:thm-sublinear} and Theorem \ref{thm:thm-linear}) can be shown with exactly the same analysis as PDHG. 

In the rest of this section, we show how to reformulate the iterate updates of PPM, ADMM and l-ADMM as an instance of \eqref{eq:cornerstone}, and present the corresponding results for the three algorithms. A mild difference between ADMM and the other algorithms is that the matrix $\mathcal P$ in ADMM is not full rank, and the inverse is not well-defined. We overcome this by defining IDS along the subspace where $\mathcal P$ is full rank.

We also would like to comment that positive semidefinite matrix $\mathcal{P}$ of different algorithms can be different. For example, $\mathcal{P}= \begin{pmatrix}
\frac 1s I_n & -A^T \\ -A & \frac 1s I_m
\end{pmatrix}$ in PDHG, $\mathcal{P}=\frac 1s I$ in PPM,   $\mathcal P=\begin{pmatrix}
            \frac{1}{\tau}I & -A^T \\ -A & \tau AA^T
        \end{pmatrix}$ in ADMM and $\mathcal P=\begin{pmatrix}
            \frac{1}{\tau}I & -A^T \\ -A & \lambda I
        \end{pmatrix}$ in l-ADMM. Consequently, the distance in the definition of IDS are chosen differently, and the trajectories of different algorithms may follow a different ellipsoid specified by the matrix $\mathcal{P}$ in Figure \ref{fig:bilinear}.

\subsection{Proximal Point Method (PPM)}
Consider a convex-concave minimax problem:
\begin{equation}\label{eq:ppm-minimax}
    \min_{x}\max_{y} \mathcal L(x,y) \ ,
\end{equation}
where $\mathcal L(x,y)$ is a potentially non-differentiable function that is convex in $x$ and concave in $y$. Proximal point method (PPM)~\cite{rockafellar1976monotone} (Algorithm \ref{alg:ppm}) is a classic algorithm for solving \eqref{eq:ppm-minimax}, and some of the other classic algorithms, such as extra-gradient method~\cite{nemirovski2004prox}, can be viewed as approximation of PPM. 
\begin{algorithm}
    \renewcommand{\algorithmicrequire}{\textbf{Input:}}
    \caption{Proximal Point Method for \eqref{eq:ppm-minimax}}
    \label{alg:ppm}
    \begin{algorithmic}[1]
        \REQUIRE Initial point $(x_0,y_0)$, step-size $s >0$.
        \FOR{$k=0,1,...$}
        \STATE $(x_{k+1},y_{k+1})=\arg\min_x\arg\max_y\left\{ \mathcal L(x,y) + \frac{1}{2s} \|x-x^k\|_2^2 -\frac{1}{2s} \|y-y^k\|_2^2 \right\}$
        \ENDFOR
    \end{algorithmic}
\end{algorithm}

Denote $\mathcal F(z)=\begin{pmatrix} \partial_x \mathcal{L}(x,y) \\ -\partial_y \mathcal L(x,y) \end{pmatrix}$. We can rewrite the update rule of PPM as follows:
\begin{equation*}
    \frac 1s(z_k-z_{k+1})\in \mathcal F(z_{k+1}) \ .
\end{equation*}

Parallel to the results of PDHG stated in Section \ref{sec:sec-decay-sublinear}-\ref{sec:linear}, we obtain:

\begin{thm}
Consider the iterates $\{z_k\}_{k=0}^{\infty}$ of PPM (Algorithm \ref{alg:ppm}) for solving a convex-concave minimax problem \eqref{eq:ppm-minimax}. Then it holds for $k\geq 1$ that:

(i) monotonically decay
\begin{equation*}
    dist^2\pran{0,\mathcal F(z_{k+1})}\leq dist^2\pran{0,\mathcal F(z_{k})} \ .
\end{equation*}
(ii) sublinear convergence
\begin{equation*}
        dist^2(0,\mathcal F(z_{k}))\leq  \frac{1}{k} \Vert z_{0}-z_* \Vert_2^2 \ .
\end{equation*}
(iii) Furthermore, assume the minimax problem \eqref{eq:ppm-minimax} satisfies metric sub-regularity condition \eqref{eq:eq-linear-condition} under $l_2$ norm on a set $\mathbb S$ that contains $\{z_k\}_{k=0}^{\infty}$, that is, for some constant $\alpha>0$,
\begin{equation*}
    \alpha dist(z,\mathcal Z^*)\leq dist(0,\mathcal F(z)) \ .
\end{equation*}
Then, it holds for any iteration $k\ge \left\lceil e/\alpha^2\right\rceil$ that
    \begin{equation*}
        dist^2(0,\mathcal F(z_k))\leq \exp\pran{1-\frac{k}{\left\lceil e/\alpha^2\right\rceil}}dist^2(0,\mathcal F(z_0)) \ .
    \end{equation*}
\end{thm}

\subsection{Linearized ADMM}

Consider a minimization problem of form:
\begin{equation}\label{eq:dual}
    \min_x\; f(x)+g(Ax) \ ,
\end{equation}
where $f$ and $g$ are convex but potentially non-smooth functions. Examples of \eqref{eq:dual} include LASSO~\cite{chen1994basis,tibshirani1996regression}, SVM~\cite{boser1992training,cortes1995support}, regularized logistic regression~\cite{koh2007interior}, image recovery~\cite{chambolle2011first,bredies2015tgv}, etc. 
The primal-dual formulation of \eqref{eq:dual} is
\begin{equation}\label{eq:minmax-admm}
    \min_{x\in \mathbb R^n}\max_{y\in \mathbb R^m} \mathcal L(x,y)=f(x)+\left \langle Ax,y \right \rangle-g^*(y) \ ,
\end{equation}
where $g^*$ is the conjugate function of $g$. Linearized ADMM (l-ADMM)~\cite{ryu2022large} is a variant of ADMM that avoids solving linear equations. Algorithm \ref{alg:lin-admm} presents l-ADMM for solving \eqref{eq:dual}. 
 
\begin{algorithm}
    \renewcommand{\algorithmicrequire}{\textbf{Input:}}
    \caption{Linearized ADMM}
    \label{alg:lin-admm}
    \begin{algorithmic}[1]
        \REQUIRE Initial point $(x_0,y_0)$, {step-size} $\lambda$, $\tau$, and conjugate function $f^*$, $g^*$.
        \FOR{$k=0,1,...$}
        \STATE $w_{k+1}=\arg\min_w\left\{ f^*(w)- \langle w,v_k-\tau(A^Ty_{k}+w_{k}) \rangle +\frac{\tau}{2}\|w-w^k\|_2^2 \right\}$
        \STATE $v_{k+1}=v_k-\tau(A^Ty_k+w_{k+1})$
        \STATE $y_{k+1}=\arg\min_y \left\{ g^*(y)- \langle A^Ty,v_{k+1}-\tau(A^Ty_k+w_{k+1})\rangle +\frac{\lambda}{2}\|y-y_k\|_2^2  \right\}$
        \ENDFOR
    \end{algorithmic}
\end{algorithm}

As the first step, we here aim to find the appropriate differential inclusion form of the update rule, that is, find a positive (semi-)definite matrix $\mathcal P$ such that the update of the algorithm can be rewritten as
\begin{equation*}
    \mathcal P(z_k-z_{k+1}) \in \mathcal F(z_{k+1}) \ ,
\end{equation*}
where $\mathcal F(z)=\begin{pmatrix}
    \partial f(x)+ A^Ty \\ \partial g^*(y) -Ax
\end{pmatrix}$.

Notice that
\begin{align}\label{eq:lin-admm-v}
        \begin{split}
            v_{k+1} & \ = v_k-\tau A^Ty_k-\tau w_{k+1} \\
            & \ = (v_k-\tau A^Ty_k) -\tau \mathrm{prox}_{f^*}^{\tau }\pran{\frac{1}{\tau}(v_k-\tau A^Ty_k)} \\
            & \ = (v_k-\tau A^Ty_k) - (v_k-\tau A^Ty_k) + \mathrm{prox}_{f}^{\frac{1}{\tau} }(v_k-\tau A^Ty_k)\\
            & \ =  \mathrm{prox}_{f}^{\frac{1}{\tau} }(v_k-\tau A^Ty_k) \ ,
        \end{split}
\end{align}
where the second equality uses the update of $w_{k+1}=\mathrm{prox}_{f^*}^{\tau }\pran{\frac{1}{\tau}(v_k-\tau A^Ty_k)}$ and the third equality follows from the basic property of proximal operator $\mathrm{prox}_{f^*}^{\tau}(t)=t-\frac{1}{\tau}\mathrm{prox}_{f}^{\frac{1}{\tau}}(\tau t)$.

Moreover, it holds that
    \begin{align}\label{eq:lin-admm-y}
        \begin{split}
            y_{k+1} & \ = \arg\min_y \left\{ g^*(y)- \langle A^T(y-y_k),v_{k+1}-\tau(w_{k+1}+A^Ty_k)\rangle +\frac{\lambda}{2}\|y-y_k\|_2^2  \right\} \\ 
            & \ = \arg\min_y \left\{ g^*(y)-  \langle y-y_k,A(2v_{k+1}-v_k) \rangle +\frac{\lambda}{2}\|y-y_k\|_2^2  \right\}\\
            & \ = \mathrm{prox}_{g^*}^{\lambda }\pran{y_k+\frac{1}{\lambda} A(2v_{k+1}-v_k)} \ .
        \end{split}
    \end{align}
Denote $\mathcal P=\begin{pmatrix}
            \frac{1}{\tau}I & -A^T \\ -A & \lambda I
        \end{pmatrix}$, $z=(v,y)$. Rearranging \eqref{eq:lin-admm-v} and \eqref{eq:lin-admm-y}, we arrive at
\begin{align*}
    \begin{split}
    \mathcal P(z_k-z_{k+1}) & \ =\begin{pmatrix}
        \frac{1}{\tau} I_n & -A^T \\ -A & \lambda I_m
    \end{pmatrix}\begin{pmatrix}
        v_k-v_{k+1} \\ y_k-y_{k+1}
    \end{pmatrix}=\begin{pmatrix}
        \frac{1}{\tau}(v_k-v_{k+1})-A^T(y_k-y_{k+1}) \\  -A(v_k-v_{k+1})+\lambda(y_k-y_{k+1})
    \end{pmatrix}\\
    & \ \in \begin{pmatrix} \partial f(v_{k+1})+A^Ty_{k+1} \\ -Av_{k+1}+\partial g^*(y_{k+1}) \end{pmatrix}=\mathcal F(z_{k+1}) \ .
    \end{split}
\end{align*}

Then, utilizing exactly the same analysis on IDS that we have developed for PDHG in Section \ref{sec:sec-decay-sublinear}-\ref{sec:linear}, we reach the results for linearized ADMM:
\begin{thm}
Consider the iterates $\{(v_k,y_k)\}_{k=0}^{\infty}$ of Linearized ADMM (Algorithm \ref{alg:lin-admm}) for solving  problem \eqref{eq:dual}. Suppose $\lambda>\tau \|A\|^2$ and denote $\mathcal P=\begin{pmatrix}
            \frac{1}{\tau}I & -A^T \\ -A & \lambda I
        \end{pmatrix}$. Then it holds for $k\geq 1$ that:

    (i) monotonically decay
    \begin{equation*}
        \mathrm{dist}^2_{\mathcal P^{-1}}(0,\mathcal F(v_{k+1},y_{k+1}))\leq\mathrm{dist}^2_{\mathcal P^{-1}}(0,\mathcal F(v_k,y_k))
    \end{equation*}

    (ii) sublinear convergence
    \begin{equation*}
        dist_{\mathcal P^{-1}}^2(0,\mathcal F(v_{k},y_k))\leq  \frac{1}{k} \Vert (v_{0},y_0)-(v_*,y_*) \Vert_{\mathcal P}^2 \ .
    \end{equation*}

    (iii) Furthermore, assume the minimax problem \eqref{eq:minmax-admm} satisfies metric sub-regularity condition \eqref{eq:eq-linear-condition} for $\mathcal P$ on a set $\mathbb S$ that contains $\{(v_k,y_k)\}_{k=0}^{\infty}$, that is, for some constant $\alpha_{\mathcal P}>0$, 
    \begin{equation*}
        \alpha_{\mathcal P} dist_{\mathcal P}((v,y),\mathcal F^{-1}(0)) \leq dist_{\mathcal P^{-1}}(0,\mathcal F(v,y)) \ .
    \end{equation*} 
    Then, it holds for any iteration $k\ge \left\lceil e/\alpha_{\mathcal P}^2\right\rceil$ that
    \begin{equation*}
        dist_{\mathcal P^{-1}}^2(0,\mathcal F(v_{k},y_k))\leq \exp\pran{1-\frac{k}{\left\lceil e/\alpha_{\mathcal P}^2\right\rceil}}dist_{\mathcal P^{-1}}^2(0,\mathcal F(v_{0},y_0)) \ .
    \end{equation*}
\end{thm}

\subsection{Alternating Direction Method of Multipliers (ADMM)}

\begin{algorithm}
    \renewcommand{\algorithmicrequire}{\textbf{Input:}}
    \caption{ADMM}
    \label{alg:admm}
    \begin{algorithmic}[1]
        \REQUIRE Initial point $(x_0,y_0)$, {step-size} $\tau >0$.
        \FOR{$k=0,1,...$}
        \STATE $w_{k+1}=\arg\min_w\left\{ f^*(w)-\langle A^Ty_k+w,v_k \rangle +\frac{\tau}{2}\|w+A^Ty_k\|_2^2 \right\}$
        \STATE $v_{k+1}=v_k-\tau(A^Ty_k+w_{k+1})$
        \STATE $y_{k+1}=\arg\min_y \left\{ g^*(y)- \langle A^Ty+w_{k+1},v_{k+1}\rangle +\frac{\tau}{2}\|A^Ty+w_{k+1}\|_2^2  \right\}$
        \ENDFOR
    \end{algorithmic}
\end{algorithm}

In this subsection, we study the property of IDS of vanilla ADMM for solving \eqref{eq:dual}. Similar to the previous sections, we first rewrite the update rule of ADMM as an instance of \eqref{eq:cornerstone}.

Notice that
\begin{align}\label{eq:admm-v}
        \begin{split}
            v_{k+1} & \ = v_k-\tau A^Ty_k-\tau w_{k+1} \\
            & \ = (v_k-\tau A^Ty_k) -\tau \mathrm{prox}_{f^*}^{\tau}\pran{\frac{1}{\tau}(v_k-\tau A^Ty_k)} \\
            & \ = (v_k-\tau A^Ty_k) - (v_k-\tau A^Ty_k) + \mathrm{prox}_{f}^{\frac{1}{\tau}}(v_k-\tau A^Ty_k)\\
            & \ =  \mathrm{prox}_{f}^{\frac{1}{\tau}}(v_k-\tau A^Ty_k) \ ,
        \end{split}
    \end{align}
where the second equality uses the update of $w_{k+1}=\mathrm{prox}_{f^*}^{\tau }\pran{\frac{1}{\tau}(v_k-\tau A^Ty_k)}$ and the third equality follows from the basic property of proximal operator $\mathrm{prox}_{f^*}^{\tau}(t)=t-\frac{1}{\tau}\mathrm{prox}_{f}^{\frac{1}{\tau}}(\tau t)$.

Furthermore, it holds that
    \begin{align}\label{eq:admm-y}
        \begin{split}
            y_{k+1} & \ = \arg\min_y \left\{ g^*(y)- \langle A^Ty+w_{k+1},v_{k+1}\rangle +\frac{\tau}{2}\|A^Ty+w_{k+1}\|_2^2  \right\} \\ 
            & \ = \arg\min_y \left\{ g^*(y) +\frac{\tau}{2}\|A^Ty+w_{k+1}-\frac{1}{\tau}v_{k+1}\|_2^2  \right\}\\
            & \ = \arg\min_y \left\{ g^*(y) +\frac{\tau}{2}\|A^Ty-A^Ty_k-\frac{1}{\tau}(2v_{k+1}-v_k)\|_{2}^2  \right\}  \ .
        \end{split}
    \end{align}
Denote $\mathcal P=\begin{pmatrix}
            \frac{1}{\tau}I & -A^T \\ -A & \tau AA^T
        \end{pmatrix}$, $z=(v,y)$. Rearranging \eqref{eq:admm-v} and \eqref{eq:admm-y}, we arrive at
\begin{align*}
    \begin{split}
    \mathcal P(z_k-z_{k+1}) & \ =\begin{pmatrix}
        \frac{1}{\tau} I_n & -A^T \\ -A & \tau AA^T
    \end{pmatrix}\begin{pmatrix}
        v_k-v_{k+1} \\ y_k-y_{k+1}
    \end{pmatrix}=\begin{pmatrix}
        \frac{1}{\tau}(v_k-v_{k+1})-A^T(y_k-y_{k+1}) \\  -A(v_k-v_{k+1})+\tau AA^T(y_k-y_{k+1})
    \end{pmatrix}\\
    & \ \in \begin{pmatrix} \partial f(v_{k+1})+A^Ty_{k+1} \\ -Av_{k+1}+\partial g^*(y_{k+1}) \end{pmatrix}=\mathcal F(z_{k+1}) \ .
    \end{split}
\end{align*}

A key difference between ADMM and the other methods we present is that $\mathcal P$ is not invertible, thus $\mathcal P ^{-1}$ in IDS is not well-defined. Fortunately, this can be overcome by redefining
IDS along the non-singular subspace of $\mathcal{P}$ as following:
\begin{mydef}\label{def:def-distance-pseudo}
The infimal sub-differential size (IDS) at solution $z$ for algorithm with update rule $\mathcal P(z_k-z_{k+1})\in\mathcal F(z_{k+1})$ for solving convex-concave minimax problem \eqref{eq:minmax} is defined as:
    \begin{equation}\label{eq:extended-IDS}
        dist_{\mathcal P^{-}}^2(0,\mathcal F(z)\cap \mathrm{range}(\mathcal P)) \ ,
    \end{equation}
where $\mathcal P^{-}$ is the pseudo-inverse of $\mathcal P$ and $\mathrm{range}(\mathcal P)$ is the range of matrix $\mathcal P$. $\mathcal F(z)=\begin{pmatrix} \partial f(x)+A^Ty \\ -Ax+\partial g(y) \end{pmatrix}$ is the sub-differential of the objective $\mathcal L(x,y)$.
\end{mydef}

Notice that \eqref{eq:extended-IDS} is still a valid progress metric of the solution, by noticing $dist^2_{\mathcal P^{-1}}(0,\mathcal F(z_k))=0$ implies that $0\in \mathcal F(z_{k+1})$. More formally, if IDS equals to zero, it holds that $\|z_k-z_{k+1}\|_{\mathcal P}^2=0$ and thus $0=\mathcal P(z_k-z_{k+1})\in \mathcal F(z_{k+1})$.

All the results we developed before holds for the refined IDS for ADMM with a similar proof, and we present the difference of the proof in Appendix \ref{app:non-singular}):
\begin{thm}
Consider the iterates $\{(v_k,y_k)\}_{k=0}^{\infty}$ of ADMM (Algorithm \ref{alg:admm}) for solving problem \eqref{eq:dual}. Denote $\mathcal P=\begin{pmatrix}
            \frac{1}{\tau}I & -A^T \\ -A & \tau AA^T
        \end{pmatrix}$. Then it holds for $k\geq 1$ that:

    (i) monotonically decay
    \begin{equation*}
        \mathrm{dist}^2_{\mathcal P^{-}}(0,\mathcal F(v_{k+1},y_{k+1})\cap \mathrm{range}(\mathcal P))\leq \mathrm{dist}^2_{\mathcal P^{-}}(0,\mathcal F(v_k,y_k)\cap \mathrm{range}(\mathcal P))
    \end{equation*}

    (ii) sublinear convergence
    \begin{equation*}
        dist_{\mathcal P^{-}}^2(0,\mathcal F(v_{k},y_k)\cap \mathrm{range}(\mathcal P))\leq  \frac{1}{k} \Vert (v_{0},y_0)-(v_*,y_*) \Vert_{\mathcal P}^2 \ .
    \end{equation*}

    (iii) Furthermore, assume the minimax problem \eqref{eq:minmax-admm} satisfies metric sub-regularity condition \eqref{eq:eq-linear-condition} for $\mathcal P$ on a set $\mathbb S$ that contains $\{z_k\}_{k=0}^{\infty}$, that is, for some constant $\alpha_{\mathcal P}>0$, 
    \begin{equation*}
        \alpha_{\mathcal P} dist_{\mathcal P}((v,y),\mathcal F^{-1}(0)) \leq dist_{\mathcal P^{-}}(0,\mathcal F(v,y)) \ .
    \end{equation*} 
    Then, it holds for any iteration $k\ge \left\lceil e/\alpha_{\mathcal P}^2\right\rceil$ that
    \begin{equation*}
        dist_{\mathcal P^{-}}^2(0,\mathcal F(v_{k},y_k)\cap \mathrm{range}(\mathcal P))\leq \exp\pran{1-\frac{k}{\left\lceil e/\alpha_{\mathcal P}^2\right\rceil}}dist_{\mathcal P^{-}}^2(0,\mathcal F(v_{0},y_0)\cap \mathrm{range}(\mathcal P)) \ .
    \end{equation*}
\end{thm}

}

\section{Conclusions and future direction}
In this paper, we introduce a new progress metric for analyzing PDHG for convex-concave minimax problems, which we call infimal sub-differential size (IDS). IDS is a natural generalization of the gradient norm to non-smooth problems. Compared to other progress metric, IDS is finite, easy to compute, and monotonically decays along the iterates of PDHG. We further show the sublinear convergence rate and linear convergence rate of IDS. The monotonicity and simplicity of the analysis suggest that IDS is the right progress metric for analyzing PDHG. All of the results we obtain here can be directly extended to PPM, ADMM and l-ADMM. A direct future direction is whether the properties of IDS hold for other primal-dual algorithms. Another future direction is whether IDS suggests good step-size and/or preconditioner for primal-dual algorithms.

\subsection*{Acknowledgement} The authors would like to thank David Applegate, Oliver Hinder, Miles Lubin and Warren Schudy for the helpful discussions that motivated the authors to identify a monotonically decaying progress metric for PDLP.


\bibliographystyle{amsplain}
\bibliography{ref-papers}

\newpage
\appendix

\blue{
\section{Analysis for ADMM}\label{app:non-singular}


In this section, we consider algorithms with update rule as follows: for some positive semi-definite matrix $\mathcal P$,  
\begin{equation}\label{alg:extend-ids}
    \mathcal P(z_k-z_{k+1})\in\mathcal F(z_{k+1}) \ .
\end{equation}
The previous results for IDS require matrix $\mathcal P$ to be positive definite. Here we extend the notion of IDS to handle positive semi-definite but not necessarily full rank $\mathcal P$. We start with the following definition of IDS. Note that for positive definite $\mathcal P$, it reduces to Definition \ref{def:def-distance}.

\begin{mydef}\label{def:IDS-ADMM}
The infimal sub-differential size (IDS) at solution $z$ for algorithm with update rule \eqref{alg:extend-ids} for solving convex-concave minimax problem \eqref{eq:minmax} is defined as:
    \begin{equation*}
        dist_{\mathcal P^{-}}^2(0,\mathcal F(z)\cap \mathrm{range}(\mathcal P)) \ ,
    \end{equation*}
where $\mathcal P^{-}$ is the pseudo-inverse of $\mathcal P$ and $\mathrm{range}(\mathcal P)$ is the range of matrix $\mathcal P$. $\mathcal F(z)=\begin{pmatrix} \partial f(x)+A^Ty \\ -Ax+\partial g(y) \end{pmatrix}$ is the sub-differential of the objective $\mathcal L(x,y)$.
\end{mydef}

There are two major differences to deal this case with singular $\mathcal P$. First, pseudo-inverse of $\mathcal P$ is used since the inverse $\mathcal P^{-1}$ is not well-defined. Furthermore, in contrast to $\mathcal F(z)$ in Definition \ref{def:def-distance}, here we consider $\mathcal F(z)\cap \mathrm{range}(\mathcal P)$ in that $\mathcal P$ is positive definite along the subspace $\mathrm{range}(\mathcal P)$ but not necessarily the whole space. 

IDS under Definition \ref{def:IDS-ADMM} is still a valid metric for optimality, namely, if $dist^2_{\mathcal P^-}(0,\mathcal F(z_k))$ is small, then $z_{k+1}$ is a good approximation to the solution. Moreover, it also keeps the desirable properties: monotonic decay, sublinear convergence and linear rate under metric sub-regularity. The results are summarized in the following propositions.

\begin{prop}
Consider algorithm with update rule \eqref{alg:extend-ids} for solving a convex-concave minimax problem \eqref{eq:minmax}. It holds for $k\geq 0$ that
\begin{equation*}
    dist_{\mathcal P^{-}}^2\pran{0,\mathcal F(z_{k+1})\cap \mathrm{range}(\mathcal P)}\leq dist_{\mathcal P^{-}}^2\pran{0,\mathcal F(z_{k})\cap \mathrm{range}(\mathcal P)} \ .
\end{equation*}
\end{prop}

\begin{proof}
The proof is similar to Proposition \ref{thm:thm-decay}. First, notice that the iterate update is given by
\begin{equation*}
    \mathcal P(z_k-z_{k+1})\in \mathcal F(z_{k+1})\ ,
\end{equation*}
and it is also obvious that
\begin{equation*}
  \mathcal P(z_k-z_{k+1})\in \mathrm{range}(\mathcal P) \ .
\end{equation*}
Hence we have
\begin{equation*}
    \mathcal P(z_k-z_{k+1})\in \mathcal F(z_{k+1}) \cap \mathrm{range}(\mathcal P) \ .
\end{equation*}

Now denote $\widetilde\omega_{k+1}=\mathcal P(z_k-z_{k+1})\in \mathcal F(z_{k+1})\cap \mathrm{range}(\mathcal P)$, then it holds for any $\omega_k\in\mathcal F(z_k)\cap \mathrm{range}(\mathcal P)$ that
\begin{align*}
    \begin{split}
        \Vert \widetilde\omega_{k+1} \Vert_{\mathcal P^{-}}^2-\Vert \omega_k \Vert_{\mathcal P^{-}}^2 & = \left \langle \widetilde\omega_{k+1}-\omega_k, \mathcal P^{-}(\widetilde\omega_{k+1}+\omega_k)\right \rangle \\
        &\leq \left \langle \widetilde\omega_{k+1}-\omega_k, 2(z_{k+1}-z_k)+\mathcal P^{-}(\widetilde\omega_{k+1}+\omega_k)\right \rangle\\
        & = \left \langle \mathcal P(z_k-z_{k+1})-\omega_k, 2(z_{k+1}-z_k)+\mathcal P^{-}(\mathcal P(z_k-z_{k+1})+\omega_k)\right \rangle\\
        & = \left \langle \mathcal P(z_k-z_{k+1})-\omega_k, (-2I+\mathcal P^{-}\mathcal P)(z_k-z_{k+1})+\mathcal P^{-}\omega_k\right \rangle\\
        & = -\Vert z_k-z_{k+1} \Vert_{\mathcal P}^2+2\omega_k^T(z_k-z_{k+1})-\Vert \omega_k \Vert_{\mathcal P^{-}}^2\\
        & = -\Vert z_k-z_{k+1} \Vert_{\mathcal P}^2+2\omega_k^T\mathcal P^-\mathcal P(z_k-z_{k+1})-\Vert \omega_k \Vert_{\mathcal P^{-}}^2\\
        & = -\Vert \mathcal P(z_k-z_{k+1})-\omega_k \Vert_{\mathcal P^{-}}^2\\
        & = -\Vert \widetilde\omega_{k+1}-\omega_k \Vert_{\mathcal P^{-}}^2 \\
        & \leq 0 \ ,
    \end{split}
\end{align*}
where the inequality utilizes $\mF$ is a monotone operator due to the convexity-concavity of the objective $\mathcal{L}(x,y)$, thus $\left \langle \widetilde\omega_{k+1}-\omega_k, z_{k+1}-z_k\right\rangle\ge 0$. The second equality uses $\widetilde\omega_{k+1}=\mathcal P(z_k-z_{k+1})$. and the fifth equality follows from $\omega_{k}\in \mathcal F(z_k)\cap \mathrm{range}(\mathcal P)$ that $\omega_k=\mathcal P^-\mathcal P\omega_k$.

Now choose $\omega_k=\arg\min_{w\in \mathcal F(z_k)\cap \mathrm{range}(\mathcal P)} \{ \|w\|_{\mathcal P^{-1}}^2\}$, and we obtain
\begin{equation*}
    dist_{\mathcal P^{-1}}^2(0,\mathcal F(z_{k+1})\cap \mathrm{range}(\mathcal P))\leq \Vert \widetilde\omega_{k+1} \Vert_{\mathcal P^{-1}}^2 \leq \Vert \omega_k \Vert_{\mathcal P^{-1}}^2 = dist_{\mathcal P^{-1}}^2(0,\mathcal F(z_k)\cap \mathrm{range}(\mathcal P)) \ ,
\end{equation*}
which finishes the proof.
\end{proof}

Provided the monotonic decay, sublinear convergence holds and thus under metric sub-regularity under $\mathcal P$-norm, IDS converges linearly. The proofs are identical to Theorem \ref{thm:thm-sublinear} and Theorem \ref{thm:thm-linear} respectively.
\begin{prop}
Consider the iterates $\{z_k\}_{k=0}^{\infty}$ of algorithm with update rule \eqref{alg:extend-ids} to solve a convex-concave minimax problem \eqref{eq:minmax}. Let $z_*\in \mathcal Z^*$ be an optimal point to \eqref{eq:minmax}. Then, it holds for any iteration {$k\geq 1$} that
\begin{equation*}
        dist_{\mathcal P^{-}}^2(0,\mathcal F(z_{k})\cap \mathrm{range}(\mathcal P))\leq  \frac{1}{k} \Vert z_{0}-z_* \Vert_{\mathcal P}^2 \ .
\end{equation*}    
\end{prop}

\begin{prop}
    Consider the iterations $\{z_k\}_{k=0}^{\infty}$ of algorithm with update rule \eqref{alg:extend-ids} to solve a convex-concave minimax problem \eqref{eq:minmax}. Suppose the minimax problem \eqref{eq:minmax} satisfies metric sub-regularity condition \eqref{eq:eq-linear-condition} on a set $\mathbb S$ that contains $\{z_k\}_{k=0}^{\infty}$, that is,
    \begin{equation*}
        \alpha_{\mathcal P}dist_{\mathcal P}(z,\mathcal Z^*) \leq dist_{\mathcal P^{-}}(0,\mathcal F(z)) \ .
    \end{equation*}
    Then, it holds for any iteration $k\ge \left\lceil e/\alpha_{\mathcal P}^2\right\rceil$ that
    \begin{equation*}
        dist_{\mathcal P^{-}}^2(0,\mathcal F(z_k)\cap \mathrm{range}(\mathcal P))\leq \exp\pran{1-\frac{k}{\left\lceil e/\alpha_{\mathcal P}^2\right\rceil}}dist_{\mathcal P^{-}}^2(0,\mathcal F(z_0)\cap \mathrm{range}(\mathcal P)) \ .
    \end{equation*}
\end{prop}

}

\end{document}